\documentclass[a4paper,reqno, 11pt]{amsart}
\usepackage[english]{babel}
\usepackage{amsmath}
\usepackage{amssymb}
\usepackage{enumerate}
\usepackage{ifthen}
\provideboolean{shownotes} 
\setboolean{shownotes}{true}
\usepackage{hyperref}
\usepackage{color}

\newcommand{\margnote}[1]{
\ifthenelse{\boolean{shownotes}}%
{\marginpar{\raggedright\tiny\texttt{#1}}}%
{}%
}
\newcommand{\hole}[1]{
\ifthenelse{\boolean{shownotes}}%
{\begin{center} \fbox{ \rule {.25cm}{0cm}
\rule[-.1cm]{0cm}{.4cm} \parbox{.85\textwidth}{\begin{center}
\texttt{#1}\end{center}} \rule {.25cm}{0cm}}\end{center}}
{}}
\newtheorem{theorem}{Theorem}[section]

\newtheorem{lemma}[theorem]{Lemma}
\newtheorem{corollary}[theorem]{Corollary}
\newtheorem{definition}[theorem]{Definition}

\theoremstyle{remark}
\newtheorem{remark}[theorem]{Remark}
\newcommand{\nada}[1]{}
		       
\newcommand{\R}{\mathbb{R}}
\newcommand{\N}{\mathbb{N}}
\newcommand{\Z}{\mathbb{Z}}
\newcommand{\T}{\mathbb{T}}
\newcommand{\ub}{u^{\be}}
\newcommand{\vn}{u^{\be,n}}
\newcommand{\wn}{u^{\be_n,n}}
\newcommand{\pn}{p^{\be,n}}
\newcommand{\omegab}{\omega^{\be}}
\newcommand{\pb}{p^{\be}}
\newcommand{\dive}{\mathop{\mathrm {div}}}
\newcommand{\curl}{\mathop{\mathrm {curl}}}
\newcommand{\be}{\alpha}
\newcommand{\fb}{f^{\be}}
\newcommand{\fub}{{\widehat{u}}^{\be}}
\newcommand{\ffb}{{\widehat{f}}^{\be}}
\newcommand{\un}{u^{{\be_n},n}}
\newcommand{\qn}{p^{{\be_n},n}}
\numberwithin{equation}{section}

\begin{document}

\title[Suitable weak solutions obtained by the Voigt
approximation]{Suitable weak solutions to the 3D Navier-Stokes
  equations are constructed with the Voigt Approximation}

\author[Berselli]{Luigi C. Berselli}
\address[L. C. Berselli]{\newline Dipartimento di Matematica
  \\
  Universit\`a di Pisa
  \\
  Via F.~Buonarroti 1/c, I-56127, Pisa, Italy}
\email[]{\href{luigi.carlo.berselli@unipi.it}{luigi.carlo.berselli@unipi.it}}

\author[Spirito]{Stefano Spirito} 

\address[S. Spirito]{\newline GSSI - Gran Sasso Science Institute
  \\
  Viale F.~Crispi 7, I-67100, L'Aquila, Italy}

\email[]{\href{stefano.spirito@gssi.infn.it}{stefano.spirito@gssi.infn.it}}

\subjclass[2010]{Primary: 35Q30, Secondary: 35A35, 76M20.}

\keywords{Navier-Stokes Equations, suitable weak solution, boundary
  value problem, Navier-Stokes-Voigt model}

\begin{abstract}
  In this paper we consider the Navier-Stokes equations supplemented
  with either the Dirichlet or vorticity-based Navier boundary
  conditions. We prove that weak solutions obtained as limits of
  solutions to the Navier-Stokes-Voigt model satisfy the local energy
  inequality. Moreover, in the periodic setting we prove that if the
  parameters are chosen in an appropriate way, then we can construct
  suitable weak solutions trough a Fourier-Galerkin finite-dimensional
  approximation in the space variables.
\end{abstract}

\maketitle

\section{Introduction}
We prove that weak solutions to the 3D Navier-Stokes
Equations~\eqref{eq:nse} (from now on NSE) obtained as limits of
solutions to the Navier-Stokes-Voigt model~\eqref{eq:Voigt} (from now
on NSV) are \textit{suitable weak solutions}.

To set the problem, we recall that the initial boundary value problem
for the incompressible NSE with unit viscosity, zero external force,
in a smooth bounded domain $\Omega\subset\R^3$ is
\begin{align}
  \label{eq:nse}
  \displaystyle{u_t-\Delta u+(u\cdot\nabla)\,u+\nabla
    p}&=0&\quad \text{in }(0,T)\times\Omega,
  \\
  \dive u&=0&\quad \text{in }(0,T)\times\Omega,
  \\
  \label{eq:bc}
  u&=0&\quad\text{on }(0,T)\times\Gamma,   
  \\
  \label{eq:id}
  u(x,0)&=u_0(x)&\qquad\text{in }\Omega,
\end{align}
where $u:\,(0,T)\times\Omega\to \R^{3}$ is the velocity vector field
and $p:\,(0,T)\times\Omega \to\R$ is the scalar pressure, and
$u_0(x)$ is a divergence-free vector initial datum. We are writing
the problem with vanishing Dirichlet boundary conditions, but we will
treat also a Navier-type boundary condition.

It is well-known that important issues as global regularity and
uniqueness of weak solutions for the 3D NSE are still open and very
far to be understood, see Galdi~\cite{Gal2000a} and Constantin and
Foias~\cite{CF1988}. A keystone regularity result for weak solutions
is the partial regularity theorem of Caffarelli, Kohn, and
Nirenberg~\cite{CKN1982} which asserts that the set of interior
(possible) singularities has vanishing one-dimensional parabolic
Hausdorff measure. Concerning partial regularity results up to the
boundary with Dirichlet boundary conditions, see~\cite{LS1999}. In
the case of Navier boundary condition~\eqref{eq:bcn} we could not find
a specific reference, however we suspect and conjecture that similar
results of partial regularity can be obtained also in this case with
minor changes.

In any case, the partial regularity theorem holds for a
particular subclass of Leray weak solutions, called starting
from~\cite{CKN1982} \textit{``suitable weak solutions.''} Beside
technical regularity properties, the most important additional
requirement of suitable weak solutions (see Scheffer~\cite{Sche1977}
and see also Sec.~\ref{sec:Pre} below for precise definitions), is the
following inequality, often called in literature \textit{local or
  generalized energy inequality}:
\begin{equation}
  \label{eq:GEI}
  \begin{aligned}
    \partial_{t}\left(\frac{1}{2}|u|^{2}\right)+\nabla\cdot\left(\left(\frac{1}{2}|u|^{2}
        +p\right)u\right)-\Delta\left(\frac{1}{2}|u|^{2}\right)
    + |\nabla u|^{2}
    \leq 0
  \end{aligned}
\end{equation}
{in }$\mathcal{D}^{'}((0,T)\times\Omega)$.

Since at present results of uniqueness for weak solutions are not
known, we cannot exclude that each method used to construct weak
solutions can produce its own class of solutions and these solutions
could not satisfy the local energy inequality. For this particular
issue see also the recent review in Robinson, Rodrigo, and
Sadowski~\cite{RRS2016}. It is then a relevant question to check
whether weak solutions obtained by different methods are suitable or
not, especially those constructed with methods which are well
established by physical or computational motivations. Together with
the construction given in~\cite{CKN1982} by retarded mollifiers, they
also recall that in~\cite{Sche1977} existence of suitable solution
(even if the name did not exist yet) has been obtained for the Cauchy
problem without external force. The technical improvements to obtain
partial regularity with external forces in the natural $L^2(\Omega)$
space --or even $H^{-1/2}(\Omega)$-- arrived only recently with the
work of Kukavica~\cite{Kuk2008} (in fact in~\cite{CKN1982} the force
needed to be in $L^{\frac{5}{3}+\delta}(\Omega)$). Here we do not
consider external forces and we are only treating the problem of
showing if the solutions constructed by certain approximations
satisfy~\eqref{eq:GEI}.

In the development of the concept of local energy inequality we recall
--in earlier times-- the two companion papers by Beir\~ao da
Veiga~\cite{Bei1985a,Bei1985b} dealing with the hyper-viscosity and a
general approximation theorem. Recent results in an exterior domain
with the Yosida approximation are also those by Farwig, Kozono, and
Sohr~\cite{FKS2005}.  Finally, we mention also that the existence of suitable weak solution has been proved by using some artificial compressibility method, see \cite{BS2016a,DS2011}.

Beside the pioneering work in~\cite{CKN1982}, the interest for the
notion of suitable solutions has been recently renewed in the context
of Large Eddy Simulation (LES) and turbulence models. %
The interest for suitable solutions comes also from the fact that the
local energy inequality seems a natural request (representing a sort
of entropy) for any reasonable approximation of the NSE. It is
especially in the field of turbulent models and LES that Guermond
\textit{et al.}~\cite{Gue2008,GPP2011,GP2005} --making a parallel with
the notion of entropy solutions-- suggested that LES models should
select \textit{``physically relevant''} solutions of the Navier-Stokes
equations, that is those satisfying the local energy inequality.

In the periodic setting it is known that most of the models belonging
to the $\alpha$-family produce suitable weak solutions, in the limit
as $\alpha\to0^+$; recall e.g. the results on Leray-$\alpha$
approximation~\cite{GOP2004}, see also the last
Sec.~\ref{sec:teo2}. The LES models are aimed at producing
approximation of the average velocity and dimensionally $\alpha$ is a
length, connected with the smallest resolved scale. Most of the LES
models are designed for the space-periodic setting, to simulate
homogeneous turbulence, and their introduction is generally considered
\textit{challenging} in presence of boundaries,
see~\cite{BIL2006,CL2014}. Among these methods the NSV (also very
close to the simplified Bardina model) seems one of the most
promising, since it does not require extra boundary conditions.  The
initial-boundary value problem for the NSV system with Dirichlet data
reads as follows: given $\be>0$ solve
\begin{align}
  \label{eq:Voigt}
  \ub_t-\be^2\Delta\ub_t-\Delta
  \ub+(\ub\cdot\nabla)\,\ub+\nabla\pb&=0&\quad\text{in
  }(0,T)\times\Omega,
  \\
  \nabla\cdot\ub&=0&\quad\text{in }(0,T)\times\Omega,
  \\
  \label{eq:bcv}
  \ub&=0&\quad\text{on }(0,T)\times\Gamma,
  \\
  \label{eq:ida}
  \ub(0,x)&=\ub_0(x)&\quad\text{in }\Omega.
\end{align}
The NSV method has been introduced by Oskolkov~\cite{Osk1973,Osk1980}
to model visco-elastic fluids, but the interest in the theory of
turbulence and LES models came with the work of Titi~\textit{et
  al.}~\cite{CLT2006,RT2010} and also with the simplified Bardina
model by Layton and Lewandowski~\cite{LL2006a}. For these reasons here
we consider the NSV model which is one of the few well-posed also in
the case of a domain with boundaries. We observe that the problem of
convergence towards a suitable weak solutions is not hard in the
space-periodic setting, see the last section. On the other hand, for
the boundary value problem the result is --as far as we know--not
solved yet and as the reader will see it requires some technical care
to be tackled. In particular, we will focus on the problems arising
from the presence boundaries and first we consider the Dirichlet case,
which is the most common. Then, we will consider another popular
(especially in turbulence problems~\cite{BIL2006,CL2014}) set of
boundary condition: the slip at the wall (Navier boundary conditions),
which are particularly interesting when studying wall effects for
non-homogeneous turbulence.  These boundary conditions, which
replace~\eqref{eq:bcv}, have been recently studied for various applied
and theoretical reasons
in~\cite{BeiC2009a,Ber2010b,BR2006,BS2012a,XX2007} and they read as
follows
\begin{equation}
  \label{eq:bcn}
  \begin{aligned}
    u\cdot n&=0\qquad\text{on }(0,T)\times\Gamma,
    \\
    \omega\times n&=0\qquad\text{on }(0,T)\times\Gamma,
  \end{aligned}
\end{equation}
where $\omega:=\text{curl}\, u$ and $n$ is the unit outward normal
vector on the boundary $\Gamma$.  The same conditions can be also used
to complement the NSV with the variables $\ub,\,\omega^\be=\curl\ub$,
hence substituting~\eqref{eq:bcv} by
\begin{equation}
  \begin{aligned}
    \label{eq:bcvn}
    \ub\cdot n&=0\qquad\text{on }(0,T)\times\Gamma,
    \\
    \omega^{\be}\times n&=0\qquad\text{on }(0,T)\times\Gamma.
  \end{aligned}
\end{equation}
We will recall the notion of solution for both problems and we will
show the differences to handle the two set of boundary conditions.

Then, we prove the main results of the paper which are the following
two theorems.
\begin{theorem}
  \label{thm:main}
  Let $\{\ub_0\}_{\be}\subset H^{1}_{0,\sigma}(\Omega)\cap
  H^{2}(\Omega)$ be a sequence of initial data converging strongly in
  $H^1_{0,\sigma}(\Omega)$ to $u_0$. Let $(\ub,\pb)$ be the
  corresponding unique weak solution of the NSV
  system~\eqref{eq:Voigt}-\eqref{eq:ida} with Dirichlet boundary
  conditions~\eqref{eq:bcv}. Then, there exists $(u,p)$ such that --up
  to a sub-sequence still labeled as $\{(\ub,\pb)\}_{\be}$-- it holds
  as $\be\to0^+$
\begin{enumerate}
\item[1)] $\ub\rightarrow u\textrm{ strongly in
  }L^{2}((0,T)\times\Omega)$,
\item[2)] $\nabla\ub\rightharpoonup \nabla u\textrm{ weakly in
  }L^{2}((0,T)\times\Omega)$,
\item[3)] $ \pb\rightharpoonup p\textrm{ weakly in
  }H^{-r}(0,T;H^{\frac{3}{10}}(\Omega))\textrm{ for all
  }r>\frac{2}{5}$,
\item[4)] The couple $(u,p)$ is a weak solution of the
  NSE~\eqref{eq:nse} and it satisfies the local energy
  inequality~\eqref{eq:GEI}.
\end{enumerate}
\end{theorem}
\begin{remark}
  It is also possible to show that the solution can be slightly
  changed to be suitable in the usual sense, see
  Corollary~\ref{cor:corollary}.
\end{remark}
\begin{theorem}
  \label{thm:main2}
  Let $\{\ub_0\}_{\be}\subset L^{2}_{\sigma}(\Omega)\cap
  H^{2}(\Omega)$ be a sequence of initial data
  satisfying~\eqref{eq:bcvn} and converging strongly in
  $L^2_\sigma(\Omega)\cap H^1(\Omega)$ to $u_0$. Let $(\ub,\pb)$ be
  the corresponding unique weak solution of the NSV
  system~\eqref{eq:Voigt}-\eqref{eq:ida} with Navier-type boundary
  conditions~\eqref{eq:bcv}. Then, there exists $(u,p)$ such that --up
  to a sub-sequence still labeled as $\{(\ub,\pb)\}_{\be}$-- it holds
  as $\be\to0^+$
  \begin{enumerate}
  \item[1)] $\ub\rightarrow u\textrm{ strongly in
    }L^{2}((0,T)\times\Omega)$,
  \item[2)] $\nabla\ub\rightharpoonup \nabla u\textrm{ weakly in
    }L^{2}((0,T)\times\Omega)$,
  \item[3)] $ \pb\rightharpoonup p\textrm{ weakly in
    }L^{\frac{5}{3}}((0,T)\times\Omega)$,
  \item[4)] The couple $(u,p)$ is a suitable weak solution of the NSE~\eqref{eq:nse}.
  \end{enumerate}
\end{theorem}
Another important open problem is whether Faedo-Galerkin approximation
methods produce solutions which are suitable or not. As far as we
know, there are only partial results of
Guermond~\cite{Gue2006,Gue2007}, concerning approximations made by a
special (but very large) class of Finite-Element spaces for the
velocity and pressure.  Anyway, the Fourier-Galerkin method (obtained
by Fourier series expansion in the space-periodic case) is a case
still not covered by the theory.  In the last section we will consider
the space-periodic case, that is $\Omega=\mathbb{T}:=(\R/2\pi\Z)^{3}$
(the three-dimensional torus), and we require that all variables have
vanishing mean value on $\Omega$. We will recall some results making
connections with the Fourier-Galerkin methods and in the final section
we also give the following result, inspired by Biryuk, Craig, and
Ibrahim~\cite{BCI2007}.
\begin{theorem}
  \label{thm:main3}
  Let  in the space-periodic setting $\{(\un,\qn)\}_{n\in\N}$ be the Fourier-Galerkin
  approximation for the system NSV~\eqref{eq:Voigt} up to the wave-number
  $|k|=n$, and corresponding to the regularization parameter
  $\be=\be_n$. If $\frac{1}{\be_n}=o\left({n^6}\right)$ then there
  exists $(u,p)$ such that, as $n\to+\infty$
\begin{enumerate}
\item[1)] $\un\rightarrow u\textrm{ strongly in }L^{2}((0,T)\times\Omega)$,
\item[2)] $\nabla \un\rightharpoonup \nabla u\textrm{ weakly in
  }L^{2}((0,T)\times\Omega)$, 
\item [3)] $\qn\rightharpoonup p\textrm{ weakly in
  }L^{\frac{5}{3}}((0,T)\times\Omega)$,
\item[4)] $(u,p)$ is a suitable weak solution of the space-periodic NSE.
\end{enumerate}
\end{theorem}

\textbf{Plan of the paper.}  In Section~\ref{sec:Pre} we are going to
fix the notation that we use in the paper and we recall the main
definition and tools used. In Section~\ref{sec:Prio} there are the a
priori estimates we will use in Section~\ref{sec:teo1} to prove
Theorem~\ref{thm:main}-\ref{thm:main2}. Finally, in
Section~\ref{sec:teo2} we prove Theorem~\ref{thm:main3}.
\section{Notation and preliminaries}
\label{sec:Pre}
We start by recalling the functional spaces we will use. Let
$\Omega\subset\R^3$ be a smooth and bounded open set. The space of
compactly supported smooth functions on $\Omega$ will be denoted by
$\mathcal{D}(\Omega)$.  We will denote with
$(L^{p}(\Omega),\|\cdot\|_p)$ the standard Lebesgue space and to
simplify the notation we will denote by $\|\cdot\|$ the
$L^2(\Omega)$-norm and the corresponding scalar product by
$(\cdot,\cdot)$.  The Sobolev space of functions with
$k$-distributional derivatives in $L^{p}(\Omega)$ is denoted by
$W^{k,p}(\Omega)$ and their norm with $\|\cdot\|_{W^{k,p}}$. As
customary we define $H^k(\Omega):=W^{k,2}(\Omega)$.  We will also use
the Bochner spaces $L^{p}(0,T;W^{k,q}(\Omega))$,
$L^{p}(0,T;H^{s}(\Omega))$ and $H^{r}(0,T;H^{s}(\Omega))$, for the
precise definitions see for instance~\cite{CF1988,Tem2001}. We use the
subscript $_{"\sigma"}$ to denote the subspace of solenoidal vector
fields, obtained by using the Leray projection operator $P$ over
tangential and divergence-free vector fields,
see~\cite{Gal2000a,Soh2001}.  In particular, it is standard to
introduce the following spaces
\begin{align*}
  &L^{2}_{\sigma}(\Omega)=\{u\in L^{2}(\Omega): \nabla\cdot
  u=0,\,\,u\cdot n=0\textrm{ on }\partial\Omega\},
  \\
  &H^{1}_{0,\sigma}(\Omega)=\{u\in H^{1}_{0}(\Omega): \nabla\cdot
  u=0,\,\,u=0\textrm{ on }\partial\Omega\}.
\end{align*}
Then, the Stokes operator $(A,D(A))$ is defined as follows:
\begin{equation}
  \label{eq:Stokes}
  \begin{aligned}
    &A:D(A)\rightarrow L^{2}_{\sigma}(\Omega),
    \\
    &D(A)=H^{1}_{0,\sigma}(\Omega)\cap H^{2}(\Omega),
    \\
    &Au=-P\Delta u,
  \end{aligned}
\end{equation}
where $-\Delta: H^{2}(\Omega)\cap H^{1}_{0}(\Omega)\rightarrow
L^{2}(\Omega)$ denotes the Laplace operator with homogeneous Dirichlet
boundary condition and $P$ denotes the projection in $L^{2}(\Omega)$
onto $L^{2}_{\sigma}(\Omega)$.  We choose the open set
$\Omega\subset\R^3$ smooth enough such that the following inequality
holds true
\begin{equation*}
  \|u\|_{H^{2}}\leq C\|Au\|\quad\textrm{ for some }C\geq 0.
\end{equation*}

We recall some technical results on the Navier-Stokes equations we
will need during the proof of the main theorems.
\subsection{Initial Value boundary problem for the NSE }
In this subsection we define the suitable weak solutions of NSE with
boundary condition~\eqref{eq:bc} or~\eqref{eq:bcn} and initial datum
$u_0\in L^2_\sigma(\Omega)$.  With Dirichlet conditions the notion of
Leray-Hopf weak solution for the NSE is well-known, see for
instance~\cite{Gal2000a}. The notion of weak solution to the NSV can
be found in~\cite{CLT2006,Lar2011}.

We recall now the definition of suitable weak solution. 
\begin{definition}
  \label{def:sws}
  Let $(u,p)$ be such that $u\in
  L^{2}(0,T;H^{1}_{0,\sigma}(\Omega))\cap
  L^{\infty}(0,T;L^{2}(\Omega))$, and $p\in
  L^{\frac{5}{3}}((0,T)\times\Omega)$, and $u$ is a Leray-Hopf weak
  solution to the Navier-Stokes equation~\eqref{eq:nse}. The pair
  $(u,p)$ is said a suitable weak solutions if the local energy
  balance~\eqref{eq:GEI} holds in the distributional sense.
\end{definition}
As we said in the introduction one of the main differences between
Leray-Hopf weak solutions and suitable weak solutions is the
inequality~\eqref{eq:GEI}. Indeed, by using the results
of~\cite{SvW1986} in the case of boundary conditions~\eqref{eq:bc} and
the Poisson equation associated to the pressure in the case of
boundary conditions~\eqref{eq:bcn}, see also Lemma~\ref{lem:pre}, it
is always possible to associate to $u$ a scalar pressure $p$ with the
regularity stated in Definition~\ref{def:sws}. However, it not know
how to prove inequality~\eqref{eq:GEI} for general Leray-Hopf weak
solutions. Indeed~\eqref{eq:GEI} is an entropy inequality and formally
is derived by multiplying the first equation of~\eqref{eq:nse} by
$u\,\phi$, with $\phi$ a positive, compactly supported test function,
and by integrating by parts in space and time. The regularity of a
weak solution is not enough to directly justify all the calculation
leading to~\eqref{eq:GEI}.

The technicalities often rely in showing certain regularity of the
pressure field. This can be done rather easily for the problem without
boundaries, since the pressure satisfies the Poisson equation
\begin{equation}
  \label{eq:ep}
  -\Delta \pb=\partial_i\partial_j(\ub_i\ub_j).
\end{equation}
In the case of periodic boundary condition or in the whole space,
from~\eqref{eq:ep} it is easy to get the necessary regularity to treat
the pressure term in~\eqref{eq:GEI}, once good estimates are known on
the velocity.

In a bounded domain with Dirichlet boundary conditions the situation
becomes more subtle. One possible approach would be that of using the
semigroup theory, in the spirit of the estimates by Sohr and von
Wahl~\cite{SvW1986} and Solonnikov~\cite{Sol1964,Sol1968}, but this
seems difficult because the approximate system NSV doesn't fit
directly in the type of equations used in that theory, being a
pseudo-parabolic system. Hence, in this paper we use a different
approach consisting of getting a pressure estimate in negative
fractional Sobolev spaces which is weaker than the classical
$L^p((0,T);L^q(\omega))$ mixed estimates available for the NSE . This
method has been successfully used by Guermond~\cite{Gue2007} to prove
that some special Galerkin methods yield suitable weak solutions. In
particular, the difficulties arising from the boundary condition can
be circumvented by using appropriate fractional powers of the Stokes
operator. It is for this reason that we need a very detailed treatment
of the Stokes operator.
\subsection{Stokes operator and negative fractional Sobolev space} 
We focus now on some special properties of the Stokes operator
$-P\Delta$. As said in the introduction in the case of Dirichlet
boundary conditions~\eqref{eq:bc} we will get uniform bound for the
pressure in negative fractional Sobolev space and we will make an
extensively use of the Stokes operator and its powers. In this
subsection we recall the main definitions and the main properties we
will use in the sequel, following very closely~\cite{Gue2007}.  Let
$\mathcal{H}$ be an Hilbert space and $L^{q}(\R;\mathcal{H})$ with
$q\geq1$ be the associated Bochner space. For $\psi \in
L^{1}(\R;\mathcal{H})$ and $\xi\in\R$ we define the Fourier transform
(with respect to the time variable) $\widehat{\psi}(\xi)$ as follows
\begin{equation*}
  \widehat{\psi}(\xi):= \int_{\R} \psi(t)\, e^{-2i\pi \xi t}\,dt.
\end{equation*}
The previous definition can be extended to the space of tempered
distribution $\mathcal{S}^{'} (\R;\mathcal{H})$. As usual we define
$\mathcal{H}^{\gamma}(\R;\mathcal{H})$ the space of tempered
distributions $v \in \mathcal{S}^{'}(\R;\mathcal{H})$ such that
\begin{equation*}
  \int_{\R} (1 + |\xi|)^{2\gamma}\,
  \|\widehat{v}(\xi)\|^{2}_{\mathcal{H}} \,d\xi<+\infty. 
\end{equation*}
The space $H^{\gamma}(0,T;\mathcal{H})$ is defined by those
distributions that can be extended to $S^{'}(\R;\mathcal{H})$ and
whose extension is in $H^{\gamma}(0,T;\mathcal{H})$. The norm in
$H^{\gamma}(0,T;\mathcal{H})$ is the quotient norm, namely,
\begin{equation*}
  \|v\|_{H^{\gamma}(0,T;\mathcal{H})} = 
  \inf_{ \begin{array}{cc}
      \scriptstyle v=u \text{{\em\ a.e. on }}(0,T)
    \end{array}} \|{v}\|_{H^{\gamma}(\R;\mathcal{H})}.
\end{equation*}
As Hilbert spaces $\mathcal{H}$ we will take the classical Sobolev
spaces $H^s(\Omega)$, allowing also for fractional values of
$s$. Since we are considering functions in a bounded domain with zero
boundary value, we need to be rather precise about the definition of
the function space, when $s\not\in\N$.

The space $H^s(\Omega)$ is defined via the real method of
interpolation as follows
\begin{equation*}
  H^s(\Omega)=\left\{
    \begin{aligned}
      &[L^2(\Omega), H^1(\Omega)]_s &\quad\text{ for }s\in(0,1),
      \\
      & [H^1(\Omega), H^2(\Omega)]_s &\quad\text{ for }s\in(1,2).
    \end{aligned}
\right.
\end{equation*}
Next, to deal with zero traces, we introduce the space $H^s_0(\Omega)$
which is the closure of $\mathcal{D}(\Omega)$ in $H^s(\Omega)$ and for
any $s\in(0,1)$. The space $\widetilde{H^s_0}(\Omega)$ is defined as
follows
\begin{equation*}
\widetilde{H}^s_0(\Omega)=\left\{
  \begin{aligned}
    &    [L^2(\Omega), H^1_0(\Omega)]_s&\qquad \text{for }s\in[0,1],
    \\
    & H^s(\Omega)\cap H^1_0(\Omega)&\qquad \text{for }s\in(1,2].
  \end{aligned}
\right.
\end{equation*}
Finally, for $s>0$ we denote with $\widetilde{H}_0^{-s}(\Omega)$ the
dual space of $\widetilde{H}^s_0(\Omega)$ and by $H^{-s}(\Omega)$ the
space defined via the norm coming from the duality
\begin{equation*}
  \|u\|_{H^{-s}}=\sup_{0\not=w\in\mathcal{D}(\Omega)}\frac{(u,w)}{\|w\|_{H^s}}. 
\end{equation*}
It is well-known that the following spaces coincide with equivalent
norms:
\begin{align*}
  &H^s(\Omega)\cong H^{s}_0(\Omega)\qquad s\in[0,1/2],
  \\
  &H^s(\Omega)\cong \widetilde{H}^{s}_0(\Omega)\qquad s\in[0,1/2),
  \\
  &H^{-s}(\Omega)\cong \widetilde{H}^{-s}_0(\Omega)\qquad
  s\in[0,1/2)\cup(1/2,3/2).
\end{align*}
Since we will not consider the critical case $s=\frac{1}{2}$ we will
always denote these spaces with $H^{s}(\Omega)$. From the
definition~\eqref{eq:Stokes} it follows that the operator $A$ is
positive and self-adjoint. By using the spectral theorem we can define
its fractional powers, specifically we can define $A^{s}$, for any
$s\in\R$, on its domain which we denote with $D(A^{s})$. We have that
the quantity $(u,A^{s}u)$ is a norm on the subspace
$D(A^{\frac{s}{2}})$. The following equivalences of norms will be
frequently used in the sequel: There exists $c_1, c_2>0$ such that
\begin{equation}
  \label{eq:norm}
  \begin{aligned}
    &c_1\|u\|_{\widetilde{H}^s}\leq (u,A^s u)^{\frac{1}{2}}\textrm{ for any }u\in
    D(A^\frac{s}{2}),\,\,s\in(-1/2,2],
    \\
    & (u,A^s u)^{\frac{1}{2}}\leq c_2\|u\|_{\widetilde{H}^s}\textrm{ for any }u\in
    D(A^\frac{s}{2}),\,\,s\in[-2,2].
  \end{aligned}
\end{equation}
For the proof of the inequalities~\eqref{eq:norm} see~\cite{GS2011}.
\subsection{The Navier-Stokes-Voigt model}
In this subsection we recall the main result regarding the
approximating system~\eqref{eq:Voigt}.
\begin{theorem}
  \label{thm:existence-Voigt}
  Let $\be>0$ be fixed and $\ub_0\in H^1_{0,\sigma}(\Omega)$. Then,
  there exists a unique weak solution $\ub\in
  L^{\infty}(0,T;H^1_{0,\sigma}(\Omega))$ and $\pb\in
  L^{2}(0,T;L^2(\Omega))$ (with norms depending on $\be>0$) of the
  initial value boundary problem~\eqref{eq:Voigt}.

  In addition, if $\ub_0\in H^{2}(\Omega)\cap
  H^1_{0,\sigma}(\Omega)$. Then, $\ub\in
  L^{\infty}(0,T;H^2(\Omega)\cap H^1_{0,\sigma}(\Omega))$ and $\pb\in
  L^{2}(0,T;H^1(\Omega))$.
\end{theorem}
We do not prove this theorem. The existence and uniqueness parts can
be found in~\cite{CLT2006} and the regularity can be proved by
standard energy estimates. For the sequel are very relevant the
estimates in terms of $\be>0$ for several norms, see
Lemmas~\ref{lem:est1}-~\ref{lem:ut}. Especially the first one is
crucial also for the existence of weak solutions. In particular,
uniform estimates on the Galerkin-approximate system allow to pass to
the limit with standard compactness results~\cite{CLT2006}. More
delicate is the question of space regularity. Since the NSV system is
pseudo-parabolic there is not an increase of regularity as for the
parabolic equations, but the solutions keep the regularity of the
initial datum, as in the hyperbolic case~\cite{Lar2011,LaTi2010}. An explicit
example of this is given in~\cite{BKR2016}. This motivates the request
that the initial datum belongs to $H^2(\Omega)$, since otherwise
calculations of the next sections would be formal and not justified.
\begin{remark}
  With a procedure of approximation the condition on the initial datum
  could be slightly relaxed.
\end{remark}
\subsection{On slip boundary conditions}
First we recall some definitions and technical facts when dealing with
the NSE and NSV with the Navier-type boundary
conditions~\eqref{eq:bcn}-\eqref{eq:bcvn}, respectively.
\begin{definition}
\label{def:weak-solution}
We say that $u\in L^{\infty}(0,T;L^{2}_\sigma(\Omega))\cap
L^{2}(0,T;H^{1}(\Omega))$, is a (Leray-Hopf) weak solution of the
NSE~\eqref{eq:nse} with boundary conditions~\eqref{eq:bcn} if the two
following hold true:
\begin{equation*}
  \begin{aligned}
    \int_{0}^{T}\int_{\Omega}\big( -u \phi_{t}+\nabla u \nabla
    \phi-(u \cdot\nabla)\, \phi\,u \big)\,dx d t 
    +\int_{0}^{T}\int_{\Gamma}
    u \cdot (\nabla n)^T\cdot \phi\,dS d t
    \\
=\int_{\Omega}u_{0}\phi(0)\,dx,
  \end{aligned}
\end{equation*}
for all vector-fields $\phi\in
C^{\infty}_0([0,T[\times\overline{\Omega})$ such that
$\nabla\cdot\phi=0$ in $[0,T[\times\Omega$, and $\phi\cdot n=0$ on
$[0,T[\times\Gamma$. Moreover, the following energy estimate
\begin{equation}
  \label{eq:Energy-NS-NS} 
  \frac{1}{2}  \|u(t)\|^{2}+\int_{0}^{t}\|\nabla u(s) \|^{2}\,ds+
  \int_{0}^{t}\int_{\Gamma}u \cdot(\nabla 
  n)^T\cdot u \ dS ds\leq\frac{1}{2}\|u_{0}\|^{2},
\end{equation}
is satisfied for all $t\in[0,T]$.
\end{definition}

With this definition we have the following result.
\begin{theorem}
  \label{thm:existence_weak_solutions}
  Let be given any positive $T>0$ and $u_0\in L^2_{\sigma}(\Omega)$,
  then there exists at least a weak solution $u $ of the Navier-Stokes
  equations~\eqref{eq:nse} on $[0,T]$.
\end{theorem}

The proof of global existence of weak solution in the sense of the
Definition~\ref{def:weak-solution} can be found for instance
in~\cite[\S~6]{XX2007}. We observe that an equivalent formulation can
be given.  To this end we recall the following formulas for
integration by parts (see~\cite{BB2009} for the proof).
\begin{lemma}
  \label{lem:gammalap}
  Let $u$ and $\phi$ be two smooth enough vector fields, tangential to
  the boundary $\Gamma$.  Then it follows
  \begin{equation*}
      -\int_\Omega \Delta u\,\phi\,dx=\int_\Omega \nabla u\nabla
      \phi\,dx-\int_{\Gamma} (\omega\times
      n)\,\phi\,dS+\int_{\Gamma}u\cdot(\nabla n)^T\cdot\phi\,dS,
  \end{equation*}
  where $\omega=\curl u$. Moreover, if $\nabla\cdot u=0$, then
  $-\Delta u=\curl\curl u$, and
 \begin{equation*}
   \int_\Omega \curl \omega\,\phi\,dx=        -\int_\Omega \Delta
   u\,\phi\,dx=\int_\Omega \omega(\curl 
   \phi)\,dx+\int_{\Gamma} (\omega\times  n)\,\phi\,dS.
 \end{equation*}
\end{lemma}
With the above formulas the weak formulation can be written as follows
\begin{equation*}
  \begin{aligned}
    \int_{0}^{T}\int_{\Omega}\big( -u \phi_{t}+\omega \curl \phi-(u
    \cdot\nabla)\, \phi\,u \big)\,dx d\tau
    =\int_{\Omega}u_{0}\phi(0)\,dx,
  \end{aligned}
\end{equation*}
for all vector-fields $\phi\in
C^{\infty}_0([0,T[\times\overline{\Omega})$ such that
$\nabla\cdot\phi=0$ in $[0,T[\times\Omega$, and $\phi\cdot n=0$ on
$[0,T[\times\Gamma$. Moreover, the following energy estimate holds
true
\begin{equation}
  \label{eq:Energy-NS-NS-bis} 
  \frac{1}{2}
  \|u(t)\|^{2}+\int_{0}^{t}\|\omega(s)\|^{2}\,ds\leq\frac{1}{2}\|u_{0}\|^{2}\qquad
  \forall\,t\in[0,T]. 
\end{equation}
Next, when we consider the NSV system with Navier conditions and we
have the following definition.
\begin{definition}
  \label{def:weak-solution-NSV}
  We say that $u\in L^{\infty}(0,T;L^2_\sigma(\Omega)\cap
  H^{1}(\Omega))$, weak solution of the NSV~\eqref{eq:Voigt} with
  boundary conditions~\eqref{eq:bcvn} if the two following condition
  hold:
  \begin{equation*}
    \begin{aligned}
      \int_{0}^{T}\int_{\Omega}\big( -\ub
      \phi_{t}-\alpha^2\omegab\curl \phi_{t}+\nabla \ub \nabla
      \phi-(\ub \cdot\nabla)\, \phi\,\ub \big)\,dx dt
      \\
      =\int_{\Omega}\ub_{0}\phi(0)+\alpha^2\nabla\ub_{0}\nabla\phi(0)\,dx,
    \end{aligned}
  \end{equation*}
  for all vector-fields $\phi\in
  C^{\infty}_0([0,T[\times\overline{\Omega})$ such that
  $\nabla\cdot\phi=0$ in $[0,T[\times\Omega$, and $\phi\cdot n=0$ on
  $[0,T[\times\Gamma$. Moreover, the following energy estimate
  \begin{equation}
    \label{eq:Energy-NSV-N} 
    \frac{1}{2}  \|\ub(t)\|^{2}+\frac{\alpha^2}{2}\|\omegab(t)\|^2+
    \int_{0}^{t}\|\omegab (s)\|^{2}\,ds
    =\frac{1}{2}\|\ub_{0}\|^{2}+\frac{\alpha^2}{2}\|\omegab_0\|^2,
  \end{equation}
  is satisfied for all $t\in[0,T]$.
\end{definition}
With this definition one can easily prove the following result (for
which we did not find any reference)
\begin{theorem}
  \label{thm:existence_weak_solutions-V}
  Let be given any $T>0$, $\be>0$ and $\ub_0\in
  L^2_{\sigma}(\Omega)\cap H^1(\Omega)$, then there exists a unique
  weak solution $\ub $ of the NSV~\eqref{eq:Voigt} with Navier
  boundary conditions~\eqref{eq:bcvn} on $[0,T]$. Moreover, if
  $\ub_0\in L^2_\sigma(\Omega)\cap H^2(\Omega)$ and $\omegab_0\times
  n=0$ at $\partial\Omega$, then the unique solution belongs also to
  $L^\infty(0,T;H^2(\Omega)\cap L^2_\sigma(\Omega))$.
\end{theorem}
The proof of this result goes through the a priori estimates obtained
testing with $\ub$ and $-P\Delta\ub$, see especially those obtained in
Lemmas~\ref{lem:est1-bis}-\ref{lem:ut-bis}.
\section{A priori estimates independent of $\be$ for solutions of the
  Navier-Stokes-Voigt model}
\label{sec:Prio}
In this section we are going to prove the main $\be$-independent
\emph{a priori} estimates needed in the proof of
Theorems~\ref{thm:main}-\ref{thm:main2}.
%
\subsection{The case of vanishing Dirichlet Boundary
  Conditions}
The first estimate we prove is the standard energy-type estimate one
obtained by multiplying equations~\eqref{eq:Voigt} by $\ub$ and by
integrating by parts over $\Omega$. We have the following result.
\begin{lemma}
  \label{lem:est1}
  Let $\ub$ be a weak solution of the
  NSV~\eqref{eq:Voigt}-\eqref{eq:ida} with initial datum $\ub_0\in
  H^1_{0,\sigma}(\Omega)$. Then, for any $t\in(0,T)$
  \begin{equation}
    \label{eq:en1}
    \|\ub(t)\|^2+\be^2\|\nabla
    \ub(t)\|^2+2\int_0^t\|\nabla\ub(s)\|^2\,ds=\|\ub_0\|^2+\be^2\|\nabla\ub_0\|^2. 
\end{equation}
\end{lemma}
Then we prove a simple weighted (in $\be$) estimate for $\ub_t$ which
will be useful to pass to the limit to get the local energy
inequality~\eqref{eq:GEI}.
\begin{lemma}
  \label{lem:ut}
  Let $\ub$ be a weak solution of the
  NSV~\eqref{eq:Voigt}-\eqref{eq:ida} with initial datum $\ub_0\in
  H^1_{0,\sigma}(\Omega)$.  Then, there exists $c>0$ independent of
  $\be$ such that for any $t\in(0,T)$
\begin{equation*}
  \be^3\|\nabla\ub\|^2+\be^3\int_0^T\|\ub_t(s)\|^2\,ds+
\be^5\int_0^T\|\nabla\ub_t(s)\|^2\,ds\leq  c.
\end{equation*}
\end{lemma}
\begin{proof}
  Let us multiply the momentum equation in~\eqref{eq:Voigt} by
  $\be^3\ub_t$. By integrating by parts over $(0,T)\times\Omega$ we
  get
  \begin{equation*}
\begin{aligned}
  \frac{\be^3}{2}\|\nabla\ub(t)\|^2+&\int_0^t\be^3\|\ub_t(s)\|^2
  +{\be^5}\|\nabla\ub_t(t)\|^2\,ds
  \\
  &\leq\be^3\int_0^t\int_{\Omega}|\ub|\, |\nabla \ub|\,|\ub_t|\,dx ds
  +\frac{\be^3}{2}\|\nabla \ub_0\|^2.
\end{aligned} 
\end{equation*}
We estimate the right-hand side by using H\"older and the standard
Gagliardo-Nirenberg type interpolation inequality (of $L^4(\Omega)$
with $L^2(\Omega)$ and $H^1_0(\Omega)$ ) as follows
\begin{align*}
  &\alpha^3 \int_0^t\int_{\Omega}|\ub| |\nabla \ub| |\ub_t|\,dx ds\leq
  \alpha^3\int_0^t\|\ub\|_{4}\|\nabla \ub\| \|\ub_t\|_{4}\,ds
  \\
  &\qquad\leq c\alpha^3\int_0^t\|\ub\|^{\frac{1}{4}}\|\nabla
  \ub\|^{\frac{7}{4}}\|\nabla
  \ub_t\|^{\frac{3}{4}}\|\ub_t\|^{\frac{1}{4}}\,ds
  \\
  &\qquad\leq c \alpha^3 (\|\ub_0\|^{2}+\alpha^{2}\|\nabla
  \ub_0\|^{2})^{\frac{1}{8}}\int_0^t\|\nabla
  \ub\|^{\frac{7}{4}}\|\nabla
  \ub_t\|^{\frac{3}{4}}\|\ub_t\|^{\frac{1}{4}}\,ds,
\end{align*}
where in the second line we use, by Lemma~\ref{lem:est1}, the fact the
$\ub$ is bounded in $L^\infty(0,T;L^2(\Omega))$ independently of
$\be$.  Now, we use Young inequality with $p_1=\frac{8}{3}$, $p_2=8$
and $p_3={2}$ and we get
\begin{align*}
  &\alpha^3 \int_0^t\int_{\Omega}|\ub||\nabla \ub||\ub_t|\,dx ds
  \\
  &\qquad\leq c(\|\ub_0\|^{2}+\alpha^{2}\|\nabla
  \ub_0\|^{2})^{\frac{1}{4}} \int_0^t\alpha^{\frac{3}{2}}\|\nabla
  \ub\|^{\frac{3}{2}}\|\nabla \ub\|^{2}\,ds+
  \frac{\alpha^3}{2}\int_0^t\|u_t\|^2\,ds
  \\
  &\qquad+\frac{\alpha^5}{2}\int_0^t\|\nabla u_t\|^2\,ds.
\end{align*}
Then, by using~\eqref{eq:en1} we have that
$\alpha^{\frac{3}{2}}\|\nabla \ub(s)\|^{\frac{3}{2}}
\leq\big(\|\ub_0\|^{2}+\alpha^{2}\|\nabla\ub_0\|^{2}\big)^{\frac{3}{4}}$,
and consequently we get
\begin{equation*}
  \begin{aligned}
    \frac{\be^{3}}{2}\|\nabla\ub(t)\|^2+&\int_0^T\be^{3}\|\ub_t(s)\|^2\,ds
    +\be^5\int_0^T\|\nabla\ub_t(s)\|^2\,ds
    \\
    &\leq\frac{\be^{3}}{2}\|\nabla\ub_0\|^2+c(\|\ub_0\|^{2}+\alpha^{2}\|\nabla
    \ub_0\|^{2}) \int_{0}^{T}\|\nabla\ub(s)\|^2\,ds
    \\
    &\leq
    \frac{\be^{3}}{2}\|\nabla\ub_0\|^2+c(\|\ub_0\|^{2}+\alpha^{2}\|\nabla
    \ub_0\|^{2})^2\leq c,
\end{aligned} 
\end{equation*}
if $0<\alpha<1$. We can suppose that $\alpha$ is always smaller than
one, since we are interested in the behavior as $\alpha\to0^+$.  In
particular, for our purposes the most relevant estimate is that there
exists a constant $c$ independent of $\alpha>0$ such that
\begin{equation}
  \label{eq:en2}
  \be^{3}\int_0^T\|\ub_t(s)\|^2\,ds\leq c.
\end{equation}
\end{proof}
\subsubsection{Estimates in fractional Sobolev space}
This section is devoted to the proof of the estimate for the pressure
in the case of Dirichlet boundary conditions. We follow the same line
of~\cite{Gue2007}.  Let be $p,q,\bar{r}$ and $s$ real positive numbers
such that the following relations hold:
\begin{equation}
  \frac{2}{p}+\frac{3}{q}=4,\quad p\in[1,2],\quad
  q\in[1,\frac{3}{2}],\quad \frac{s}{3}:=\frac{1}{q}-\frac{1}{2},\quad
  \bar{r}:=\frac{1}{p}-\frac{1}{2}. 
\label{4.1}
\end{equation}
If $p,q$ satisfy~\eqref{4.1} we have by Sobolev embedding that
\begin{equation*}
  L^{p}(0,T;L^{q}(\Omega))\subset H^{-r}(0,T;H^{-s}(\Omega)).
\end{equation*}
The first lemma we recall is an estimate for the nonlinear term in
negative-fractional Sobolev spaces. See~\cite[Lemma 3.4]{Gue2007}.
\begin{lemma}
  \label{lem:non}
  Let $\ub$ be a weak solution to NSV, then for any
  $s\in\left[\frac{1}{2},\frac{3}{2}\right]$, there exists a constant
  $c>0$, independent of $\be$ such that
  \begin{equation}
    \label{eq:NL}
    \| (\ub\cdot\nabla)\,\ub\|_{H^{-r}(0,T;H^{-s})}\leq c.
  \end{equation}
\end{lemma}
\begin{proof}
  By Sobolev embedding and interpolation inequality we have that if
  $\ub\in L^\infty(0,T;L^2(\Omega))\cap L^2(0,T;H^1(\Omega))$, then by
  duality
\begin{equation}
  \label{eq:l1}
  \|(\ub\cdot\nabla)\,\ub\|_{L^{p}(0,T;H^{-s})}\leq c,
\end{equation}
for any $p\in [1,2]$ and $s\in\left[\frac{1}{2},\frac{3}{2}\right]$
such that $\frac{2}{p}+s=\frac{7}{2}$.  In particular, in the above
lemma is exactly the same which is valid for Leray-Hopf weak solutions
to the NSE~\eqref{eq:nse}.

Then, we extend $(\ub\cdot\nabla)\,\ub$ to $0$ out of $(0,T)$ and we
take the Fourier transform with respect to the time variable. By using
H\"older inequality and Hausdorff-Young inequality we
get~\eqref{eq:NL}.
\end{proof}
Then, we use this information on the convective term, when considered
as a right-hand side, to infer further properties of
$(\ub,p^\be)$. This is more or less the same approach as
in~\cite{Gue2007} and it is based on an extension of classical results
on fractional derivatives. The relevant point is that one can use
Hilbert-space techniques at the price of working with negative
norms. The following Lemma is a refined estimate of the velocity in
fractional Sobolev spaces.
\begin{lemma}
\label{lem:s} 
For any $\chi\in\left[\frac{1}{4},\frac{1}{2}\right)$ and
$\tau<\bar{\tau}=\frac{2}{5}(1+\chi)$ there exists $c>0$, independent
of $\be$, such that
\begin{equation}
  \label{eq:4.7}
  \|\ub\|_{H^{\tau}(0,T;H^{-\chi})}\leq c.
\end{equation}
\end{lemma}
\begin{proof}
  We write the system~\eqref{eq:Voigt} in the following way
  \begin{align}
    \label{eq:app2}
    \ub_t-\be^2\Delta\ub_t-\Delta
    \ub+\nabla\pb&=-(\ub\cdot\nabla)\,\ub&\qquad\text{in
    }(0,T)\times\Omega,
    \\
    \nabla\cdot\ub&=0&\qquad\text{in }(0,T)\times\Omega.
  \end{align}
  By applying $P$ to the equations~\eqref{eq:app2} we get
  \begin{equation}
    \label{eq:app3}
    \ub_t+\be^2A\ub_t+A\ub=-P((\ub\cdot\nabla)\,\ub)\qquad\text{in }(0,T).
  \end{equation}
  where $A$ is the Stokes operator.  Since we are going to use Fourier
  transform with respect to time we need to extend all the functions
  from $[0,T]$ to $\R$. We extend $\ub$ by $(t+1)\ub_{0}$ on $[-1,0]$
  and by $0$ on $[T+1,\infty)$. We denote this extension by
  $\bar{u}^{\be}$.  Let $\varphi \in C^{\infty}(\R)$ be such that
  $\text{supp}(\varphi)\subset(-1,T+1)$ and $\varphi\equiv 1$ on
  $[0,T]$, we denote with a slight abuse of notation
  \begin{equation*}
    \ub=\varphi \bar{u}^{\be}.
  \end{equation*}
  Next, we define the following function
  \begin{equation*} 
    f^{\be} = \left\{
      \begin{array}{cc}
        (1+t)\varphi^{'}(t)(I+\be^2A) \ub_{0}+\varphi(t)
        (I+\be^2A)\ub_{0}
        \\-\varphi(t)(1+t)A\ub_0 & \textrm{  }t\in (-1,0), 
        \\
        \\
        -\varphi(t)P((\ub\cdot\nabla)\,\ub)+\varphi^{'}(t)(I+\be^2
        A)\ub         & \textrm{  }t\not\in (-1,0). 
      \end{array} 
    \right.
\end{equation*}
It follows that $\ub$ e $f^{\be}$ are well defined on
$(-\infty,+\infty)$.  Then,~\eqref{eq:app3} becomes
\begin{equation}
  \label{eq:app3bis}
  \ub_t+\be^2A\ub_t+A\ub=\fb\qquad \text{in }t\in\R.
\end{equation}
By using~\eqref{eq:NL} with $s=\frac{3}{2}$ we get that for any $r>0$
there exists $c>0$, independent of $\be$, such that
\begin{equation}
  \label{eq:fb}
  \|f^{\be}\|_{H^{-r}(0,T;{H}^{-\frac{3}{2}})}\leq c.
\end{equation}
Next, we take the Fourier transform of~\eqref{eq:app3bis} with respect
to $t$ we get the following (abstract) equation in $L^2(\Omega)$
\begin{equation} 
  \label{eq:app4}
  2\pi i\,\xi(\fub+\be^2A\fub)+A\fub=P\ffb,
\end{equation}
and it is at this point that we require the initial datum in
$H^2(\Omega)$ in order that $A\fub$ is well defined in $L^2(\Omega)$.  Let $\chi$ as
in the statement and let us take the $L^2(\Omega)$-scalar product of
the equations~\eqref{eq:app4} and $A^{-\chi}\fub$.  We then obtain
\begin{equation*}
  2\pi
  i\,\xi\,[(\fub,A^{-\chi}\fub)+\be^2(A\fub,A^{-\chi}\fub)]+(A\fub,A^{-\chi}\fub)
  =(P\ffb,A^{-\chi}\fub).    
\end{equation*}
Then, since $A$ is self-adjoint and positive, $(A\fub,A^{-\chi}\fub)$
is real and non-negative. By taking the imaginary part of both sides we get
\begin{equation*}
  |\xi|\,|(\fub,A^{-\chi}\fub)|\leq c|(P\ffb,A^{-\chi}\fub)|.
\end{equation*}
Since $\chi<\frac{1}{2}$ we can use the norm
equivalence~\eqref{eq:norm} and
\begin{equation*}
  \|\fub\|_{H^{-\chi}}^2\leq c(\fub,A^{-\chi}\fub). 
\end{equation*}
Concerning the right-hand side of~\eqref{eq:app4} by using
again~\eqref{eq:norm} we have
\begin{equation*}
  \begin{aligned}
    \|A^{-\chi}\fub\|_{H^\frac{3}{2}}^2&\leq (A^{-\chi}\fub,A^\frac{3}{2}A^{-\chi}\fub)
    \\
    &\leq (\fub,A^{\frac{3}{2}-2\chi}\fub)
\leq c\|\fub\|^2_{H^{\frac{3}{2}-2\chi}}.
  \end{aligned}                                         
\end{equation*}                                        
Then, we get 
\begin{equation}
  \label{eq:vef1}
  |\xi|\|\fub\|_{H^{-\chi}}^2\leq
  C\|\ffb\|_{H^{-\frac{3}{2}}}\|\fub\|_{H^{\frac{3}{2}-2\chi}}.
\end{equation}
Note that for all $\chi\in\left[\frac{1}{4},\frac{1}{2}\right)$ we have that 
\begin{equation*}
  -\chi<\frac{3}{2}-2\chi<1,
\end{equation*}
hence we can interpolate as follows $H^{\frac{3}{2}-2\chi}(\Omega)$
between $H^{-\chi}(\Omega)$ and $H^1(\Omega)$
\begin{equation}
  \label{eq:vef2}
  \|\fub\|_{{H}^{\frac{3}{2}-2\chi}}\leq
  \|\fub\|_{{H}^{-\chi}}^{\gamma}\|\fub\|_{{H}^{1}}^{1-\gamma}, 
\end{equation}
with $\gamma=\frac{4\chi-1}{2+2\chi}$.  Inserting~\eqref{eq:vef2}
in~\eqref{eq:vef1} we have
\begin{equation}
  \label{eq:4.12}
  |\xi|\|\fub\|_{H^{-\chi}}^{2-\gamma}\leq
  c\|\widehat{f}^\be\|_{H^{-\frac{3}{2}}}\|\fub\|_{H^{1}}^{1-\gamma}, 
\end{equation}
Since $L^{2}(\Omega)\subset H^{-\chi}(\Omega)$ for any $\chi>0$, for
$\gamma\in(0,1)$ we have the following inequality
\begin{equation}
  \label{eq:4.12bis}
  \|\fub\|^{2-\gamma}_{H^{-\chi}}\leq\|\fub\|^{2-\gamma}_{H^{1}}.
\end{equation}
By summing up~\eqref{eq:4.12} and~\eqref{eq:4.12bis} we get
\begin{equation}
  (1+|\xi|)\|\fub(\xi)\|^{2-\gamma}_{H^{-\chi}}
  \leq\|\widehat{f}^\be(\xi)\|_{H^{-\frac{3}{2}}}\|\fub(\xi)\|_{H^{1}}^{1-\gamma}+  
  \|\fub(\xi)\|^{2-\gamma}_{H^{1}}. 
\end{equation}
Let $r>0$ and set $\nu=\frac{2r}{2-\gamma}$. By dividing both sides by
$(1+|\xi|)^\nu$ we have
\begin{equation}
  \label{4.12}
  \begin{aligned}
    & (1+|\xi|)^{\frac{2}{2-\gamma}-\nu}\|\fub(\xi)\|_{H^{-\chi}}^{2}
    \\
    &\qquad\leq
    c(1+|\xi|)^{-\nu}\|\ffb(\xi)\|_{H^{-s}}^{\frac{2}{2-\gamma}}
    \|\fub(\xi)\|_{H^{1}}^{\frac{2(1-\gamma)}{2-\gamma}}+
    \|\fub(\xi)\|_{H^{1}}^2.
  \end{aligned}
\end{equation}
By integrating~\eqref{4.12} with respect to $\xi\in\R$ and by using
H\"{o}lder inequality we get
\begin{equation*}
  \begin{aligned}
    &
    \int_{\R}(1+|\xi|)^{\frac{2(1-r)}{2-\gamma}}\|\fub(\xi)\|_{H^{-\chi}}^{2}\,d\xi
    \\
    &\qquad \leq c
    \|\ffb\|_{H^{-r}(0,T;H^{-\frac{3}{2}})}^{\frac{1}{2-\gamma}}\|\fub\|_{L^{2}(0,T;H^{1})}^{\frac{(1-\gamma)}{2-\gamma}}
    +\|\fub\|_{L^{2}(0,T;H^{1})}^2.
  \end{aligned}
\end{equation*}
Note that since $r>0$ we have 
\begin{equation*}
  \tau=\frac{1-r}{2-\gamma}<\frac{1}{2-\gamma}=\frac{2}{5}(1+\chi)=\bar{\tau}.
\end{equation*}
By using that $\ub\in L^2(0,T;H^1(\Omega))$, which becomes by
Plancherel theorem $\|\fub(\xi)\|_{H^{1}}\in L^2(\R)$,
and~\eqref{eq:fb} finally we get~\eqref{eq:4.7}.
\end{proof}
Once we have an estimate on $\ub$ in fractional spaces, we can derive
a corresponding estimate for $\Delta \ub$, by the properties of the
Stokes operator.
\begin{lemma}
  For all $s\in\left[\frac{1}{2},\frac{3}{2}\right)$ and $r>\bar{r}$,
  there exists $c$ independent of $\be$ such that
  \begin{equation}
    \label{eq:4.6}
    \|\Delta\ub\|_{H^{-r}(0,T;{H}^{-s})}\leq c.
  \end{equation}
\end{lemma}
\begin{proof}
  First we have that 
  \begin{equation}
    \label{eq:norm2}
    \|\Delta \ub\|_{H^{-s}}^2\leq c\|A\ub\|_{H^{-s}}^2.
  \end{equation}
  Note that since $\frac{1}{2}<2-s<\frac{3}{2}$ we could
  use~\eqref{eq:norm}.  We multiply~\eqref{eq:app4} by $A^{1-s}\fub$ and
  we get, taking now only the real part,
  \begin{equation*}
    \|A\fub(\xi)\|_{H^{-s}}^{2}\leq c \|\ffb(\xi)\|_{H^{-s}}\|A\fub(\xi)\|_{H^{-s}} .
  \end{equation*}
  By simplifying the square of $\|A\fub\|_{H^{-s}}$,
  using~\eqref{eq:norm2} and integrating in time we get~\eqref{eq:4.6}.
\end{proof}
Finally we come back to the equations without the Leray projection
over divergence-free vector fields. We prove by comparison an estimate
for the pressure, which will be used to prove Theorem~\ref{thm:main}.
\begin{lemma}
\label{lem:p}
For any $r>\frac{2}{5}$ there exists $c>0$, independent of $\be$, such that 
\begin{equation}
  \label{eq:3.23}
  \|\pb\|_{H^{-r}(0,T;H^{\frac{3}{10}})}\leq c.
\end{equation}
\end{lemma}
\begin{proof}
 We come back to the~\eqref{eq:Voigt} and let us start to estimate
  the term with the Laplacian of $\ub_t$.
\begin{equation}
  \label{eq:3.24}
  \begin{aligned}
  \|\be^2\Delta\ub_t\|_{H^{-r}(0,T;H^{-\frac{7}{10}})}&\leq
    \|\be^2A\ub_t\|_{H^{-r}(0,T;H^{-\frac{7}{10}})}
    \\
    &\leq \|\ub_t\|_{H^{-r}(0,T;H^{-\frac{7}{10}})}+\|A\ub\|_{H^{-r}(0,T;H^{-\frac{7}{10}})}
    \\
    &+\|\fb\|_{H^{-r}(0,T;H^{-\frac{7}{10}})}
    \\
    &\leq c+\|\ub\|_{H^{1-r}(0,T;H^{-\chi})},
  \end{aligned}
\end{equation}
where $\chi<\frac{1}{2}$ and we have used~\eqref{eq:norm},~\eqref{eq:NL}
and~\eqref{eq:4.6}.  Then, we have
\begin{equation}
  \label{eq:3.25}
  \begin{aligned}
    \|\pb\|_{H^{-r}(0,T;H^{\frac{3}{10}})}&\leq\|\nabla
    p\|_{H^{-r}(0,T;H^{-\frac{7}{10}})}
    \\
    &\leq
    \|\ub_t\|_{H^{-r}(0,T;H^{-\frac{7}{10}})}+\|\Delta\ub\|_{H^{-r}(0,T;H^{-\frac{7}{10}})}
    \\
    &+\|\fb\|_{H^{-r}(0,T;H^{-\frac{7}{10}})}+\|\be^2\Delta\ub_t\|_{H^{-r}(0,T;H^{-\frac{7}{10}})}
    \\
    &\leq c+2\|\ub\|_{H^{1-r}(0,T;H^{-\chi})}.
  \end{aligned}
\end{equation}
In order to estimate $\ub$ in~\eqref{eq:3.25} we are going to use
Lemma~\ref{lem:s}. We have to find
$\chi\in\left[\frac{1}{4},\frac{1}{2}\right)$ such that
\begin{equation*}
1-r<\bar{\tau}=\frac{2}{5}(1+\chi),
\end{equation*}
namely, $\frac{3}{2}-\frac{5}{2}r<\chi.$ This is actually always
possible because
\begin{equation*}
\frac{3}{2}-\frac{5}{2}r<\frac{1}{2},
\end{equation*}
then we can always find $\chi\in\left[\frac{1}{4},\frac{1}{2}\right)$ such that 
\begin{equation*}
\frac{3}{5}-\frac{5}{2}r<\chi<\frac{1}{2}.
\end{equation*}
With this choice of $\chi$ by Lemma~\ref{lem:s} we
get~\eqref{eq:3.23}.
\end{proof}
\subsection{Navier boundary conditions}
\label{sec:teo3}
In this section we prove the estimates needed to prove
Theorem~\ref{thm:main2} that concerns the problem with the Navier
boundary conditions~\eqref{eq:bcn}.  We report here the following
estimates, which are counterpart of those obtained in the Dirichlet
case.
\begin{lemma}
  \label{lem:est1-bis}
  Let $\ub$ be a solution of the NSV with Navier condition, then in
  addition to the estimate~\eqref{eq:Energy-NSV-N} we have also that
  \begin{equation}
    \label{eq:en1-bis}
    \|\ub(t)\|^2+\be^2\|\nabla\ub(t)\|^2+\int_0^t\|\nabla\ub(s)\|^2\,ds\leq
    c\big(\|\ub_0\|^2+\be^2\|\nabla\ub_0\|^2\big),  
  \end{equation}
  where the constant $c$ depends only on $\Omega$.
\end{lemma}
\begin{proof}
  The proof follows easily by observing that Lemma~\ref{lem:gammalap}
  implies that
  \begin{equation*}
    \|\omegab\|^2=\|\nabla \ub\|^2+\int_\Gamma \ub\cdot\nabla n\cdot
    \ub\,dS, 
  \end{equation*}
  hence by trace theorems
\begin{equation*}
  \|\nabla \ub\|^2\leq    \|\omegab\|^2+c\int_\Gamma
  |\ub|^2\,dS\leq \|\omegab\|^2+\frac{1}{2}\|\nabla\ub\|^2+ c
  \|\ub\|^2, 
\end{equation*}
Then substituting and by using the estimate for $\ub \in
L^\infty(0,T;L^2(\Omega))$ coming from the definition of weak solution
we have the thesis.
\end{proof}
We then prove a simple weighted (in $\be$) estimate for $\ub_t$ we
will use to pass to the limit in the local energy inequality.
\begin{lemma}
  \label{lem:ut-bis}
  Let $\ub$ be a solution of~\eqref{eq:Voigt}. Then, there exists
  $c>0$ independent of $\be$ such that for any $t\in(0,T)$  
\begin{equation*}
  \be^3\|\nabla\ub\|^2+\be^3\int_0^t\|\ub_t(s)\|^2\,ds+
\be^5\int_0^t\|\nabla\ub_t(s)\|^2\,ds\leq
  c.
\end{equation*}
\end{lemma}
\begin{proof}
  The proof is very similar to that of Lemma~\ref{lem:ut}. We start by
  multiplying the momentum equation in~\eqref{eq:Voigt} by
  $\be^3\ub_t$ and by integrating by parts we get
  \begin{equation*}
    \begin{aligned}
      \frac{
        \be^3}{2}\|\omegab(t)\|^2+\int_0^t\be^3\|\ub_t(s)\|^2&
      +{\be^5}\|\omegab_t(t)\|^2\,ds 
      \\
      &\leq\be^3\int_0^t\int_{\Omega}|\ub|\, |\nabla \ub|\,|\ub_t|\,dx ds
      +\frac{\be^3}{2}\|\omegab_0\|^2.
    \end{aligned} 
  \end{equation*}
  We estimate the right-hand side by using H\"older inequality, the standard
  convex interpolation,  and the Sobolev embedding
  $H^{1}(\Omega)\subset L^{6}(\Omega)$ we get
  \begin{equation*}
    \|\ub\|_{4}\leq c
    \|\ub\|^{\frac{1}{4}}\|\ub\|_{6}^{\frac{3}{4}}\leq
    c(\|\ub\|+\|\ub\|^{\frac{1}{4}}\|\ub\|^{\frac{3}{4}}).
  \end{equation*}
  By using also Lemma~\ref{lem:gammalap}, the basic energy type
  inequality, and the previous Lemma~\ref{lem:est1-bis} we obtain the
  thesis.
  
  Again, for our purposes, the most relevant estimate is the bound
  (independent of $\alpha$) 
  $\alpha^{3/2}u_{t} \in L^{2}(0,T;L^{2}(\Omega))$.
\end{proof}
The main difference with respect to the Dirichlet case is the
treatment of the pressure, which is now much simpler. In particular,
the use of the Navier-type conditions allow us to infer the following
lemma, see~\cite{BeiC2009a,BS2012a}.
\begin{lemma}
  \label{lem:zeta}
  Let $v$ be a smooth vector field satisfying $(\curl v)\times n=0$ on
  $\Gamma$. Then, $\zeta=\curl\curl v$ is a vector field tangential to
  the boundary, i.e., $\zeta\cdot n=0$. In particular in our case,
  since $\nabla\cdot\ub=0$, then we have $\curl\curl \ub=-\Delta \ub$
  in $\Omega$. Moreover since $\ub\cdot n=\ub_t\cdot n=0$ on $\Gamma$,
  we finally get that
  \begin{equation*}
    \Delta \ub\cdot n=\Delta \ub_t\cdot n=0\quad\text{on }\Gamma. 
  \end{equation*}
\end{lemma}
For a detailed proof see~\cite[Lemma~7.4]{Ber2010b}. In that reference
many other results on the Navier conditions are reviewed.

Let $\mathcal{U}\subset\R^3$ be a neighbourhood of $\Gamma$ and
$n:\R^3\to\R^3$ be a smooth function with compact support in
$\mathcal{U}$ and such that $n\big|_{\Gamma}$ is the normal vector to
$\Gamma$.
\begin{lemma}
  \label{lem:pre}
  Let $(\ub,\pb)$ be a smooth solution to the NSV
  system~\eqref{eq:Voigt} with boundary
  condition~\eqref{eq:bcvn}. Then $\pb$ satisfy the following Neumann
  problem
  \begin{equation}
    \label{eq:lp}
    \begin{cases}
      -\Delta\pb&=\partial_i\partial_j(\ub_i\ub_j)\qquad\text{in
      }\Omega,
      \\
      \frac{\partial p}{\partial n}&= \ub_i\nabla_j n_i
      \ub_i\qquad\text{on }\partial\Omega.
    \end{cases}
  \end{equation}
  Consequently, there exists $c>0$, independent of $\be$, such that the
  following estimate holds true for all $t\in(0,T)$
  \begin{equation}
    \label{eq:spn}
    \int_0^t\|\pb(s)\|^{\frac{5}{3}}_{{\frac{5}{3}}}\,ds\leq c.
  \end{equation}
\end{lemma}
\begin{proof}
  By taking the divergence of the momentum equation we get
  \begin{equation*}
    -\Delta\pb=\dive(\ub\cdot\nabla)\,\ub,
  \end{equation*}
  and by classical interpolation inequality we have
  $(\ub\cdot\nabla)\,\ub$ is uniformly bounded in
  $L^{\frac{5}{3}}(0,T;L^{\frac{15}{14}}(\Omega))$ with respect to
  $\be$.  This holds true because we are using only that $\ub\in
  L^\infty(0,T;L^2(\Omega))\cap L^2(0,T;H^1(\Omega))$ to obtain the
  estimate.  (The regularity inherited by Leray-Hopf weak solutions)

  By multiplying the momentum equation (restricted to $\Gamma$) by $n$
  and by using the fact that $\ub\cdot n=0$ on $(0,T)\times\Gamma$ we
  get
  \begin{equation*}
    \frac{\partial\pb}{\partial n}=\Big(\Delta\ub\cdot
    n+\be^2\Delta\ub_t-u_t-(\ub\cdot\nabla)\, \ub\Big)\cdot
    n=\ub_i\ub_j\partial_j n_i, 
  \end{equation*}
  where we have used Lemma~\ref{lem:zeta} and that on
  $(0,T)\times\Gamma$ the equality $u\cdot\nabla u\cdot
  n=-u\cdot\nabla n\cdot u$ holds true, see~\cite{BS2012a}.  Then,
  $(\ub,\pb)$ satisfies~\eqref{eq:lp}. By using a trace theorem and
  the fact that $\ub_i\ub_j$ is in
  $L^{\frac{5}{3}}(0,T;W^{1,\frac{15}{14}}(\Omega))$ we have that
  $\ub_i\ub_j\partial_j n_i \in
  L^{\frac{5}{3}}(0,T;W^{1-\frac{14}{15}}(\Gamma))$ uniformly in
  $\be$. Then, by classical $L^p$ estimates for the scalar Neumann
  problem, see~\cite{ADN1959,Tem1975} we get
  \begin{equation*}
    \|\nabla \pb\|_{L^{\frac{5}{3}}(0,T;L^{\frac{15}{14}})}\leq c,
  \end{equation*}
  with $c>0$ independent of $\be$. By using a Sobolev embedding
  inequality, and since $\pb$ is with zero mean value, we finally
  get~\eqref{eq:spn}.
\end{proof}
\section{Local energy inequality and the Proofs of
  Theorems~\ref{thm:main}-\ref{thm:main2}}
\label{sec:teo1}
In this section we prove the convergence of $\{(\ub,\pb)\}_\be$ to a
suitable weak solutions of the NSE when $\be\rightarrow 0$. Note that,
the passage to the limit in the weak formulation of the NSV to show
that the limit satisfy the NSE is standard. Specifically, either in
the case of boundary condition~\eqref{eq:bc} or in the
case~\eqref{eq:bcn} it is possible to prove that there exists $u\in
L^{\infty}(0,T;L^{2}(\Omega))\cap L^{2}(0,T;H^1(\Omega))$ a Leray-Hopf
weak solution such that, up to sub-sequences,
\begin{equation}
  \label{eq:conv}
  \begin{aligned}
    &\ub\rightarrow u\textrm{ strongly in }L^{2}((0,T)\times\Omega),
    \\
    &\nabla\ub\rightharpoonup\nabla u\textrm{ weakly in
    }L^{2}((0,T)\times\Omega). 
  \end{aligned}
\end{equation}
Then, we need only to prove the the local energy inequality holds
true. At this point, due to the fact that in the definition of local
energy inequality there test functions which are with compact support,
the role of boundary conditions is limited. We wish to mention that in
case of regularity results up to the boundary slightly different
notions are used, see the work of Seregin~\textit{et al.}  summarized
in~\cite{SS2014}.
\begin{proof}[Proof of Theorem~\ref{thm:main}]
  We start by multiplying equations~\eqref{eq:Voigt} by $\ub\phi$ for
  some $0\leq\phi\in C_c^\infty((0,T)\times\Omega)$, and after some
  integration by parts over $(0,t)\times\Omega$ we get
  \begin{equation}
    \label{5.2}
    \begin{aligned}
      \int_{0}^{T}(|\nabla\ub|^{2},\phi)\,dt&=\int_{0}^{T}\big(\frac{|\ub|^{2}}{2},\phi_{t}+
      \Delta\phi\big)+(\ub\frac{|\ub|^{2}}{2},\nabla\phi)\,dt
      \\
      &\qquad+\int_0^T \be^{2}(\Delta\ub_t,\ub\phi) +(\ub\pb,\nabla\phi)\,dt.
    \end{aligned}
  \end{equation}
  We estimate the terms from the right-hand side of~\eqref{5.2}: By
  weak lower semicontinuity of the $L^{2}$-norm, the fact that
  $\ub\rightharpoonup u$ weakly in $L^{2}(0,T;{H}^{1})$, and
  $\phi\geq0$ we have that
  \begin{equation*}
    \int_{0}^{T} (|\nabla u|^{2},\phi)\,dt\leq \liminf_{\be\rightarrow
      0}\int_{0}^{T}  (|\nabla \ub|^{2},\phi)\,dt. 
  \end{equation*}
  Then, since $\ub\rightarrow u$ strongly in
  $L^{2}(0,T;L^{2}(\Omega))$, we get
  \begin{equation}
    \int_{0}^{T}\left(\frac{|\ub|^{2}}{2},\phi_{t}+\Delta\phi\right)\,dt\rightarrow
    \int_{0}^{T}\left(\frac{|u|^{2}}{2},\phi_{t}+
      \Delta\phi\right)\,dt\qquad\textrm{as }\be\rightarrow 0. 
    \label{5.4}
  \end{equation}
  Next, by interpolation we also have that $\ub\rightarrow u$ strongly
  in $L^{2}(0,T;L^{3}(\Omega))$ and that $\ub$ is bounded in
  $L^{4}(0,T;L^{3}(\Omega))$, so it follows that 
  \begin{equation*}
    \int_{0}^{T}
    (\ub\frac{|\ub|^{2}}{2},\nabla\phi)\,dt\rightarrow\int_{0}^{T}(u\frac{|u|^{2}}{2},\phi)\,dt 
    \qquad\textrm{as } \be\rightarrow 0. 
  \end{equation*}
  of~\eqref{5.2}. To estimate the third term, we observe that since
  $\phi$ is with space-time compact support in $(0,T)\times \Omega$ we
  can freely integrate by parts without appearance of boundary
  terms. With a first integration by parts with respect to the time
  variable and then with respect to the space variables we get
  \begin{equation*}
    \begin{aligned}
      &\int_0^T\int_\Omega\Delta \ub_t \ub \phi\,dx dt=-
      \int_0^T\int_\Omega\Delta \ub \ub_t \phi+ \Delta \ub \ub
      \phi_t\,dx dt
      \\
      &\quad=\int_0^T\int_\Omega\nabla \ub \nabla \ub_t \phi+ \nabla
      \ub \nabla\phi \,\ub_t+|\nabla \ub|^2\phi_t+\nabla \ub \ub
      \nabla\phi_t\,dx dt
      \\
      &\quad=\int_0^T\int_\Omega\partial_t\frac{|\nabla
        \ub|^2}{2}\phi+ \nabla \ub \nabla\phi\,\ub_t+|\nabla
      \ub|^2\phi_t+\nabla\frac{|\ub|^2}{2} \nabla\phi_t\,dx dt.
    \end{aligned}
  \end{equation*}
  Hence, with further integration by parts of the first and last term
  \begin{equation*}
    \begin{aligned}
      &\int_0^T\int_\Omega\Delta \ub_t \ub \phi\,dx dt=
      \\
      &\quad=\int_0^T\int_\Omega-\frac{|\nabla \ub|^2}{2}\phi_t+ \nabla \ub \nabla\phi
      \ub_t+|\nabla \ub|^2\phi_t-\frac{|\ub|^2}{2} \Delta\phi_t\,dx dt
      \\
      &\quad=\int_0^T\int_\Omega\frac{|\nabla \ub|^2}{2}\phi_t+ \nabla \ub \nabla\phi
      \ub_t+|\nabla \ub|^2\phi_t-\frac{|\ub|^2}{2} \Delta\phi_t\,dx dt.
    \end{aligned}
  \end{equation*}
  Consequently, we can prove the following estimate
  \begin{align*}
    & \be^2\left|\int_0^T\int_\Omega\Delta\ub_t\ub\phi\,dx dt\right|
    \\
    &\qquad \leq
    \frac{\be^2}{2}\int_0^T\int_\Omega|\nabla\ub|^2|\phi_t|\,dx dt
    +\frac{\be^2}{2}\int_0^T\int_\Omega|\ub|^2|\Delta\phi_t|\,dx dt
    \\
    &\qquad\qquad
    +\be^2\left|\int_0^T\int_\Omega\ub_t\nabla\ub\nabla\phi\,dx dt\right|.
  \end{align*}
  By using the fact that $\ub$ and $\nabla\ub$ are uniformly bounded
  in $L^2((0,T)\times\Omega)$ and that $\phi$ is
  $C_c^\infty((0,T)\times\Omega)$ we have that the first two integrals
  from the right-hand side vanish when $\be\rightarrow 0$. Concerning
  the last term we argue in the following way
  \begin{align*}
    \be^2\left|\int_0^T\int_\Omega\ub_t\nabla\ub\nabla\phi\,dx dt\right|&\leq
    C\be^2\int_0^T\|\ub_t\|\|\nabla\ub\|\,dt
    \\
    &\leq
    C\be^\frac{1}{2}\int_0^T\be^\frac{3}{2}\|\ub_t\|\|\nabla\ub\|\,dt
    \\
    &\leq
    C\be^\frac{1}{2}\left(\be^3\int_0^T\|\ub_t\|^2\,dt\right)^\frac{1}{2}\left(\int_0^T\|\nabla
      u\|^2\,dt\right)^\frac{1}{2}
    \\
    &\leq C\be^\frac{1}{2},
  \end{align*}
  where we have used Lemma~\ref{lem:ut}. By letting $\be$ go to $0$
  all these integrals vanish. 
  \begin{remark}
    All arguments used up to now are true in the case of both boundary
    conditions~\eqref{eq:bc} and~\eqref{eq:bcn}.
  \end{remark}
  Now, we estimate the last two term from the right-hand side of~\eqref{5.2}.
  \begin{equation}
    \label{eq:pressure}
    \int_{0}^{T} (\ub\pb,\nabla \phi)\,dt.
  \end{equation}
  In the case of the Dirichlet boundary condition~\eqref{eq:bc} by
  using Lemma~\ref{lem:p} we have that
  \begin{equation*}
    \pb\rightharpoonup p\quad \textrm{  in  }\quad
    H^{-r}(0,T;H^{\frac{3}{10}}(\Omega)). 
  \end{equation*}
  for any $r>\frac{2}{5}$. Since $\frac{1}{4}<\frac{3}{10}$ we have
  that $H^{-\frac{1}{4}}(\Omega)\subset H^{-\frac{3}{10}}(\Omega)$
  with compact embedding. Moreover,with $\chi=\frac{1}{2}$ we have
  that $\bar{\tau}=\frac{2}{5}(1+\chi)=\frac{1}{2}$. Then, we can find
  $r$ and $\tau$ such that
  \begin{equation*}
    \frac{2}{5}<r<\tau<\frac{1}{5}.
  \end{equation*}
  With this choice of parameters we have that 
  \begin{equation*}
    H^{\tau}(0,T;H^{-\frac{1}{4}}(\Omega))\subset
    H^{r}(0,T;H^{-\frac{3}{10}}(\Omega)). 
  \end{equation*}
  Then, by using Lemma~\ref{lem:s} and classical compactness argument (a
  variant of Aubin-Lions lemma, see~\cite{Gue2007}) we get that
  \begin{equation*}
    \ub\rightarrow u\textrm{ in }H^{r}(0,T;H^{-\frac{3}{10}}(\Omega)). 
  \end{equation*} 
  Then, it follows that
  \begin{equation*}
    \int_{0}^{T}(\pb\ub,\nabla\phi)\,dt\rightarrow\int_{0}^{T}(p\,u,\nabla\phi)\,dt  
    \qquad\text{as } \be\rightarrow 0
  \end{equation*}
  and this proves that the local energy inequality~\eqref{eq:GEI}
  holds true.
\end{proof}
In the Theorem~\ref{thm:main} the local energy inequality is
satisfied, but the pressure has not the usual regularity. We show now
how to slightly change the pressure to obtain a genuine suitable
solution with pressure in $L^{\frac{5}{3}}((0,T)\times\Omega)$.
\begin{corollary}
\label{cor:corollary}
  It is possible to associate to $p$ another scalar $\tilde{p}\in
  L^{\frac{5}{3}}((0,T)\times\Omega)$ such that the couple
  $(u,\tilde{p})$ is  a suitable weak solution.
\end{corollary}
\begin{proof}
  We start by recalling that In Sohr and von Wahl~\cite{SvW1986} (see
  also~\cite{CKN1982}) it is proved that the pressure associated to
  Leray-Hopf weak solution belongs to the space
  $L^{\frac{5}{3}}((0,T)\times\Omega)$. This is obtained by
  considering the linear Stokes problem
\begin{align}
  \label{eq:se}
  \displaystyle{v_t-\Delta v+\nabla
    q}&=-(u\cdot\nabla)\,u&\quad \text{in }(0,T)\times\Omega,
  \\
  \dive v&=0&\quad \text{in }(0,T)\times\Omega,
  \\
  \label{eq:sbc}
  v&=0&\quad\text{on }(0,T)\times\Gamma,   
  \\
  \label{eq:sid}
  v(x,0)&=u_0(x)&\qquad\text{in }\Omega,
\end{align}
and by employing classical $L^p$-estimates, together with the
uniqueness for the linear problem.  We observe that in
Theorem~\ref{thm:main} $p\in H^{-r}(0,T;H^{\frac{3}{10}}(\Omega))$,
while considering \eqref{eq:se} with right-hand side
$-(u\cdot\nabla)\,u$ we get that 
\begin{equation*}
q\in L^{\frac{5}{3}}(0,T)\times\Omega).
\end{equation*}
Then it follows that $p$ and $q$ are almost the same, in particular we
have 
\begin{equation*}
  \int_0^T\int_\Omega(\nabla p-\nabla q)\,\phi\,dxdt=0,
\end{equation*}
implying that 
\begin{equation*}
  q(t,x)=p(t,x)+\mathcal{G}(t)
\end{equation*}
for some function $\mathcal{G}(t)$ depending only on the time
variable.  It follows that the couple $(u,q)$ is again a weak solution
to the NSE with the same initial datum as $(u,p)$.
Next, we consider the approximate NSV problems 
\begin{equation*}
\begin{aligned}
  \ub_t-\be^2\Delta\ub_t-\Delta
  \ub+(\ub\cdot\nabla)\,\ub+\nabla(\pb+\mathcal{G})&=0&\hspace{-0cm}\text{in
  }(0,T)\times\Omega,
  \\
  \nabla\cdot\ub&=0&\hspace{-0cm}\text{in }(0,T)\times\Omega,
  \\
  \ub&=0&\hspace{-0cm}\text{on }(0,T)\times\Gamma,
  \\
  \ub(0,x)&=\ub_0(x)&\hspace{-0cm}\text{in }\Omega,
\end{aligned}
\end{equation*}
for which we can prove exactly the same estimates as before. Since the
function $\mathcal{G}(t)$ has (at least) the same time-regularity as
$p$. we can prove with the same arguments as before that
  Now, we estimate the last two term from the right-hand side of~\eqref{5.2}.
  \begin{equation}
    \int_{0}^{T} (\ub(\pb+\mathcal{G}(t)),\nabla \phi)\,dt\to
    \int_{0}^{T} (u(p+\mathcal{G}(t)),\nabla \phi)\,dt, 
  \end{equation}
hence proving that $(u,p+\mathcal{G}(t))$ satisfy the local energy inequality.
\end{proof}

\medskip

In the case of Navier boundary conditions the proof of the local
energy inequality is very similar to the previous case, only the
convergence of the term with the pressure requires a different
treatment
\begin{proof}[Proof of Theorem~\ref{thm:main2}]
  In the case of boundary conditions~\eqref{eq:bcn} the passage to the
  limit $\be\to0$ is very standard. The proof is the same as in
  Theorem~\ref{thm:main} up to the treatment of the pressure. In this
  case we have that by using Lemma~\ref{lem:pre}
  \begin{equation*}
    \pb\rightharpoonup p\textrm{ weakly in
    }L^{\frac{5}{3}}((0,T)\times\Omega), 
  \end{equation*}
  while by standard interpolation argument
  \begin{equation*}
    \ub\rightarrow u\textrm{ strongly in
    }L^{\frac{5}{2}}((0,T)\times\Omega). 
  \end{equation*}
  Then, it straightforward to pass to the limit in the
  term~\eqref{eq:pressure}.
\end{proof}
\section{Fourier-Galerkin approximations and suitable weak solutions} 
\label{sec:teo2}
In this section we consider the NSE~\eqref{eq:nse} and
NSV~\eqref{eq:Voigt} equations in the space-periodic setting and we
address the problem of construction of suitable solutions by means of
the Fourier-Galerkin method. The standard (Fourier) Galerkin method to
approximate system in $\T:=\R^3/(2\pi \Z)^3$ can be implemented in the
following way. Let $P$ denote the Leray projector of $L^2_{0}(\T)$
(the subspace of $L^2(\T)$ with zero mean value) onto divergence free
vector fields denoted by $L^2_{\sigma}(\T)$, which explicitly reads in
the orthogonal Hilbert basis of complex exponentials as follows:
\begin{equation*}
  P:\ g(x)=\sum_{k\in\mathbb{Z}^3\backslash\{0\}}g_k\, \text{e}^{
    i  k\cdot x}\mapsto Pg(x)=\sum_{k\in\mathbb{Z}^3\backslash\{0\}}\left[g_k-\frac{(g_k\cdot
      k) k}{|k|^2}\right]\, \text{e}^{
    i k\cdot x}. 
\end{equation*}
It is well-known that the Leray projector is continuous also as an
operator $H^s(\T)\cap L^2_0(\T)\mapsto H^s_\sigma(\T):=H^s(\T)\cap
L^2_s(\T)$ for all positive $s$.  For any $n\in\N$, we denote by $P_n$
the projector of $L^2(\T)$ on the finite-dimensional sub-space
$V_n:=P_n (L^2_\sigma(\mathbb{T}))$ given by the following formula
\begin{equation*}
  P_n:g(x)=\sum_{k\in\mathbb{Z}^3\backslash\{0\}}g_k \,\text{e}^{
    ik\cdot
    x}\mapsto P_n g(x)=\sum_{0<|k|\leq n}\left[g_k-\frac{(g_k\cdot
      k) k}{|k|^2}\right]\, \text{e}^{
    i k\cdot x}. 
\end{equation*}
Then the Galerkin approximation for the NSE~\eqref{eq:nse} is the
following Cauchy problem for systems of ordinary differential
equations in the unknowns $c_k^n(t)$, with $|k|\leq n$
\begin{align}
  \label{eq:GNSE}
  u^n_t- \Delta u^n+P_n((u_n\cdot\nabla)\,u_n)&=0&\quad \text{in
  }(0,T)\times\T,
  \\
  u^n(0,x)&=P_nu_0(x)&\quad \text{in }\T,
\end{align}
where 
\begin{equation*}
  u^n(t,x)=\sum_{0<|k|\leq n}c^{n}_k(t)\,\text{e}^{
    ik\cdot x},\qquad\text{with}\quad k\cdot c_k^n=0.
\end{equation*}
It is important to point to that in the standard Fourier-Galerkin
method there is an explicit formula for the Leray projector and even
if the pressure $p^{n}$ disappears, it can explicitly computed.

One main unsolved question is to prove (or disprove) that
$\{(u^n,p^{n})\}_n$ converges as $n\to+\infty$ to a suitable weak
solution. This special setting, with an approximation which is
spectral and consequently \textit{non local} seems to require tools
completely different from those successfully used in~\cite{Gue2006} to
handle finite element approximations. Some conditional results linking
the hyper-dissipative NSE and the problem of Fourier-Galerkin
approximation are treated in~\cite{BCI2007}. In that reference there
is also an interesting link between the \textit{global energy
  equality} and the local energy inequality~\eqref{eq:GEI}.

By following the same spirit, we consider then the approximation by a
Fourier-Galerkin-NSV system and show that, under a link between the
coefficient $\be$ and the order of approximation $n\in\N$ one can show
the local energy inequality. We use the symbol $\vn$ to denote the
approximate Fourier-Galerkin solution to the NSV, in the unknowns
$d_{k}^{n}(t)$ with modes up to $|k|\leq n$
\begin{equation}
  \label{eq:GA}
  \begin{aligned}
    \vn_t-\be^2\Delta \vn_t- \Delta
    \vn+P_n((u_n\cdot\nabla)\,u_n)&=0&\hspace{-.3cm}\text{ in }(0,T)\times\T,
    \\
    \vn(0,x)&=P_nu_0(x) &\text{ in }\T,
  \end{aligned}
\end{equation}
where 
\begin{equation*}
  \vn(t,x)=\sum_{0<|k|\leq n}d_k^{n}(t)\,\text{e}^{
    ik\cdot x},\qquad\text{with}\quad k\cdot d^{n}_k=0.
\end{equation*}
The Galerkin approximation for the NSV~\eqref{eq:Voigt} is the
following Cauchy problem for systems of ordinary differential
equations

The estimates which may give to the local energy inequality are
obtained by testing the equations\eqref{eq:GA} by $\vn \phi$, where
$\phi$ which is a non-negative, space-periodic, and with compact
support in $(0,T)$. In general $\vn\phi\not\in V_n$, hence we cannot
use directly this approach. We have two possible choices to project
$\vn$, or to rewrite the equations~\eqref{eq:GA} in such a way to have
the pressure and a formulation which avoids the projection over
$V_n$. It is natural, since we are in the periodic case, to define
$\pn$ as solution of the Poisson problem
\begin{equation}
  \label{eq:AP} 
  -\Delta \pn=\sum_{i,j=1}^{3}\partial_i\partial_j(\vn_{i}\vn_{j}),
\end{equation} 
endowed with periodic conditions and normalized with vanishing mean
value.  After having defined the operator $Q_n:=P-P_{n}$ we get the
following equations for $\vn$
\begin{equation*}
  \begin{aligned}
    \vn_{t}-\be_n^2\Delta \vn_{t}- \Delta
    \vn+&P((\vn\cdot\nabla)\,\vn)
    \\
    &-Q_n((\vn\cdot\nabla)\,\vn)=0\quad\text{ in }(0,T)\times\T,
  \end{aligned}
\end{equation*}
which can be rewritten in the following way
\begin{equation}
  \label{eq:GA1}
  \begin{aligned}
    \vn_t-\be^2&\Delta \vn_t- \Delta \vn+(\vn\cdot\nabla)\,\vn
    \\
    &-Q_n((\vn\cdot\nabla)\,\vn)+\nabla \pn=0\quad\text{ in
    }(0,T)\times\T.
  \end{aligned}
\end{equation}
We can now freely test~\eqref{eq:GA1} with $\vn\phi$, but at the price
of being able to obtain good estimates on
$Q_n((\vn\cdot\nabla)\,\vn)$. It is at this step that in the finite
element setting that a special choice of the function spaces can be
used to prove the local energy inequality. In the Fourier-Galerkin
setting it is not known whether this methodology works or not, since
the available estimates are not strong enough to handle the remainder
term involving $Q_n((\vn\cdot\nabla)\,\vn)$.
\subsection{A priori estimates}
in this section we prove the main weighted estimates needed to prove
the local energy inequality under certain assumptions on the parameter
$\alpha$ of the Voigt regularization.  We begin with the standard a
priori estimate on the solution $(\vn, \pn)$.
\begin{lemma}
  \label{lem:2.0}
  Let $u_0\in H^{1}_\sigma(\T)$. Then, for all $\be>0$ and $n\in\N$,
  the unique solution of~\eqref{eq:GA} satisfies for all $t>0$ the
  following equality
  \begin{equation*}
    \|\vn(t)\|^2+\be^2\|\nabla \vn(t)\|^2+2\int_0^t\|\nabla
    \vn(s)\|^2\,ds=\|P_nu_0\|^2+\be^2\|\nabla P_n u_0\|^2. 
  \end{equation*}
\end{lemma}
This lemma is just the standard energy type equality, which is --by
the way-- satisfied also in the limit $n\to+\infty$ by weak solutions
of the NSV.

The following Lemma gives a weighted a priori estimate for second
order space derivatives and will be used to pass to the limit to get
the generalized energy inequality.
\begin{lemma}
  \label{lem:2.1}
  Let be given $u_{0}\in H^{2}_\sigma(\T)$ and let $\vn$ be the unique
  solution of~\eqref{eq:GA} with initial datum $P_{n}u_{0}$. Then,
  there exists $c>0$, independent of $\be>0$ and of $n\in \N$, such
  that for all $t\in(0,T)$
\begin{equation*}
  \be^{6}\int_0^t\|\Delta \vn(s)\|^2\,ds\leq c.
\end{equation*}
\end{lemma}
\begin{proof}
  We multiply by $-\Delta\vn$ the equations~\eqref{eq:GA} and after an
  integration by parts we get (recall that due to the choice of the
  basis $\Delta\vn\in V_{n}$)
  \begin{equation*}
    \begin{aligned}
      \frac{1}{2}\frac{d}{dt}(\|\nabla\vn\|^2+&\be^2\|\Delta\vn\|^2)+\|\Delta\vn\|^2 
      \\
      &
      =\int_\Omega P_n((\vn\cdot\nabla)\,\vn)\cdot\Delta\vn\,dx
      \\
      &=\int_\Omega (\vn\cdot\nabla)\,\vn\cdot P_n(\Delta\vn)\,dx
      \\
      &=\int_\Omega (\vn\cdot\nabla)\,\vn\cdot \Delta\vn\,dx.
    \end{aligned}
  \end{equation*}
  Integrating in time over $(0,t)$ and since the initial datum belongs
  to $H^2(\T)$ we get
  \begin{equation*}
    \begin{aligned}
      \|\vn\|^2+&
      \be^2\|\Delta\vn\|^2+2\int_0^t\|\Delta\vn(s)\|^2\,ds\leq
      \|u_{0}\|^2+\be^2\|\Delta u_{0}\|^2
      \\
      &\  +2\int_0^t\int_\Omega |\vn||\nabla\vn| |\Delta\vn|\,dx ds
      \\
      &\leq \|u_{0}\|^2+\be^2\|\Delta u_{0}\|^2
      \\
      &\
      +c\int_0^t\|\vn(s)\|^{\frac{1}{4}}\|\nabla\vn(s)\|\|\Delta\vn(s)\|^{\frac{7}{4}}\,ds
      \\
      &\leq \|u_{0}\|^2+\be^2\|\Delta u_{0}\|^2
      +c\sup_{0<t<t}\|\nabla\vn(t)\|^{6}\int_0^t\|\nabla\vn(s)\|^{2}\,ds,
    \end{aligned}
  \end{equation*}
  where we have used the usual Gagliardo-Nirenberg, H\"older, and
  Young inequalities and Lemma~\ref{lem:2.0}. By multiplying the above
  inequality on both side by $\be^6$ and using again
  Lemma~\ref{lem:2.0} we get that, for all $0<\alpha\leq1$ and for all
  $t\in(0,T)$,
  \begin{equation*}
    \be^6\|\nabla\vn(t)\|^2+\be^8\|\Delta\vn(t)\|^2+\be^6\int_0^t\|\Delta
    \vn(s)\|^2\,ds\leq c, 
  \end{equation*}
  with a constant $c$ independent of $\alpha$ and of $n$, thus ending
  the proof.
\end{proof}
Moreover, we have also the analogue of Lemma~\ref{lem:ut}, we omit the
proof since it is essentially the same.
\begin{lemma}
  \label{lem:2}
  Let $u_0\in H^{1}_\sigma(\T)$ and let $\vn$ be the corresponding
  solution of~\eqref{eq:GA}. Then, there exists $c>0$, independent of
  $\be>0$ and of $n\in \N$, such that for all $t\in(0,T)$
  \begin{equation*}
    \be^{3}\int_0^t\|\partial_t\vn(s)\|^2\,dt\leq c.
  \end{equation*}
\end{lemma} 
From the elliptic equations associated to $\pn$ we can easily prove
the following estimate
\begin{lemma}
  \label{lem:2.3}
  Let $u_0\in H^{1}_\sigma(\T)$ and let $(\vn, \pn)$ be a solution
  of~\eqref{eq:GA}, where the pressure is defined
  through~\eqref{eq:AP}. Then, there exists $c>0$, independent of
  $\be>0$ and of $n\in \N$, such that for all $t\in(0,T)$ such that
  \begin{equation*}
  \int_0^t\|\pn(s)\|_{{ \frac{5}{3}}}^{\frac{5}{3}}\,ds\leq c.
\end{equation*}
\end{lemma} 
We observe that from the previous estimates it follows that in the
space-periodic case we can prove that the unique weak solution
$(\ub,p^\be)$ of the space-periodic NSV~\eqref{eq:Voigt} obtained as
$\lim_{n\to+\infty}(\vn,\pn)$ satisfies the same estimates. Hence, we
can easily infer the following result, which we recalled in the
introduction.
\begin{theorem}
  Let in the space-periodic case $(\ub,p^{\be})$ be a weak solution to
  the NSV equation~\eqref{eq:Voigt}. Then, as $\be\to0$ the couple
  $(\ub,p^{\be})$ converges (up to sub-sequences) to a suitable weak
  solution to the space-periodic NSE~\eqref{eq:nse}.
\end{theorem}
\begin{proof}
  The proof is apart some simplifications the same as in the case of
  Navier conditions of Theorem~\ref{thm:main2}. In some sense the use
  of Navier-conditions allows to use, even in a slightly more
  complicated way, the same tools of reconstructing the pressure via
  the solution of a Poisson equation.
\end{proof}
Concerning the Galerkin approximation, the challenging point is
studying the limit obtained in the other way around: first the limit
as $\be\to0$ and then that as $n\to\infty$. This is still an open
problem.  A possible way to partially handle this problem is to link
$\be$ and $n$ as stated in Theorem~\ref{thm:main3}. To this end we
first recall the following lemma, which is proved as one of the steps
in~\cite[Lemma~4.4]{BCI2007}.
\begin{lemma}
  \label{lem:2.4}
  Let us define
  \begin{equation*}
    g(t):=\|Q_n(\phi(t)\, \vn(t))\|_{L^\infty(\T)},
  \end{equation*}
  where $\phi\in C^{\infty}((0,T)\times\T)$ is space-periodic and with
  support contained in $(0,T)$. Then, there exists a constant $c$
  depending only on $\phi$ such that if $\vn(t,x)=\sum_{0<|k|\leq
    n}d_k^{n}(t)\,\text{e}^{ik\cdot x}$, then
  \begin{equation*}
    g(t)^2\leq
    c\left(n^2\sum_{|k|\geq\frac{n}{2}}|d^{n}_k(t)|_2^2+
      \frac{1}{n}\sum_{k\in\mathbb{Z}^3}|d^{n}_k(t)|_2^2\right).   
  \end{equation*}
\end{lemma}
This lemma will be used to estimate the integral
\begin{equation*}
  \int_0^T\int_{\T}  Q_n((\vn\cdot\nabla)\,\vn)\, \vn\phi\,dx dt,
\end{equation*}
which comes out when testing~\eqref{eq:GA1} by $\vn\phi$.
\subsection{Proof of the Theorem~\ref{thm:main3}}
First, we note that the convergence $\lim_{n\to+\infty}\wn\to u$
towards a Leray-Hopf weak solutions is standard. In particular, we get
from the basic a priori estimate and Lemma~\ref{lem:2.3} that there
exists $u\in L^{\infty}(0,T;L^2(\T))\cap L^{2}(0,T;H^1(\T))$ such that
the following convergences hold true
\begin{equation}
  \label{eq:conv-bis}
  \begin{aligned}
    &\wn \rightarrow u\textrm{ strongly in }L^2(0,T;L^2(\T)),
    \\
    &\wn\rightharpoonup u\textrm{ weakly in }L^2(0,T;H^1(\T)),
    \\
    &\qn\rightharpoonup p\textrm{ weakly in
    }L^{\frac{5}{3}}(0,T;L^{\frac{5}{3}}(\T)).
  \end{aligned}
\end{equation}
As in the proof of Theorem~\ref{thm:main} we need to prove that the
local energy inequality is satisfied. We multiply the
equations~\eqref{eq:GA1} by $\wn\phi$ with $\phi$ non-negative
positive and belonging to $C_c^\infty((0,T)\times\T)$. After standard
integrations by parts and by using the fact that $Q_n$ is a projector
on the orthogonal of $V_n$ we get
\begin{equation*}
  \begin{aligned}
    & \int_{0}^{T}(|\nabla \wn|^{2},\phi)\,dt
    \\
    &\qquad=\int_{0}^{T}\frac{|\wn|^{2}}{2}(\phi_{t}+\Delta\phi)+
    (\wn\frac{|\wn|^{2}}{2},\nabla\phi) +(\wn\,\pn,\nabla\phi)\,dt
    \\
    &\qquad\qquad+\be_n^2\int_0^T(\Delta\wn_{t},\wn\phi)\,dt-
    \int_0^T((\wn\cdot\nabla)\,\wn, Q_n(\wn\phi))\,dt.
\end{aligned}
\end{equation*}
By using the convergences from Eq.~\eqref{eq:conv-bis},
Lemmas~\ref{lem:2}-\ref{lem:2.3}, and the results of the previous
section, it is possible to show how pass to the limit as $n\to+\infty$
in all the terms of the above equality, except the last one. To
successfully handle this we have to assume a particular behavior for
the sequence $\{\be_n\}_{n}$. Indeed, by using H\"older inequality we
have that
\begin{equation*}
  \begin{aligned}
    &\left|\int_0^T((\wn\cdot\nabla)\,\wn,Q_n( \wn\phi))\,dt\right|
    \\
    & \qquad\qquad \leq\int_0^T\|\wn(t)\|\|\nabla
    \wn(t)\|\|Q_n(\wn\phi)(t)\|_{L^{\infty}}\,dt
    \\
    &\qquad\qquad \leq\|\wn(t)\|_{L^\infty(0,T;L^2)}\|\nabla
    \wn(t)\|_{L^2(0,T;L^2)}\|g(t)\|_{L^2(0,T)}
    \\
    & \qquad\qquad\leq c\left(\int_0^T
      n^2\sum_{|k|\geq\frac{n}{2}}|\wn_k(t)|^2+\frac{1}{n}
      \sum_{k\in\mathbb{Z}^3}|{u}_{n}^{k}(t)|^2\,dt\right)^{\frac{1}{2}}.
  \end{aligned} 
\end{equation*}
where we have used the fact that, by the basic energy estimate of
Lemma~\ref{lem:2.0}, both $\|\wn\|_{L^\infty(0,T;L^2)}$ and $\|\nabla
\wn\|_{L^2(0,T;L^2)}$ are uniformly bounded in $n\in\N$. Moreover, the
second term in the right-hand side of the above inequality converges
to $0$ as $n\rightarrow \infty$ because the sum is bounded, again by
using the standard a priori estimate.

Then, we have only to show that the term
\begin{equation*}
  \int_0^Tn^2\sum_{|k|\geq\frac{n}{2}}|\wn_k(t)|^2\,dt,
\end{equation*}
converges to zero, as $n\to+\infty$.

In particular, by using the a-priori estimates from
Lemma~\ref{lem:2.1} we have that
\begin{equation*}
\begin{aligned}
  \int_0^T n^2\sum_{|k|\geq\frac{n}{2}}|\wn_k(t)|_2^2&=
  \frac{n^2\be_n^6}{n^2\be_n^6}\int_0^T\sum_{|k|\geq\frac{n}{2}}
  n^2|\wn_k(t)|^2\,dt
  \\
  &\leq 4
  \frac{\be_n^6}{n^2\be_n^6}\int_0^T\sum_{|k|\geq\frac{n}{2}}|k|^4|\wn_k(t)|^2\,dt
  \\
  &\leq \frac{4}{n^2\be_n^6}\
  \be_n^6\int_0^T\sum_{k\in\Z^3\backslash\{0\}}|k|^4|\wn_k(t)|^2\,dt
  \\
  &\leq \frac{c}{n^2\be_n^6} \ \be_n^6\int_0^T\|\Delta
  \wn(t)\|^2\,dt\leq \frac{c}{n^2\be_n^6} .
\end{aligned}
\end{equation*}
Then, for any positive sequence $\{\be_n\}_n$ such that
\begin{equation*}
  \lim_{n\to+\infty}\be_n=0\qquad\text{and}\qquad
  \lim_{n\to+\infty}\be_n^6n^2=0
\end{equation*}
we get that
\begin{equation*}
  \int_0^T g^2(t)\,dt\rightarrow +\infty,
\end{equation*}
 then the generalized energy inequality is proved.
 \section*{Acknowledgments}
 The research that led to the present paper was partially supported by
 a grant of the group GNAMPA of INdAM.
 \bibliographystyle{amsplain}
 \def\ocirc#1{\ifmmode\setbox0=\hbox{$#1$}\dimen0=\ht0 \advance\dimen0
   by1pt\rlap{\hbox to\wd0{\hss\raise\dimen0
       \hbox{\hskip.2em$\scriptscriptstyle\circ$}\hss}}#1\else {\accent"17 #1}\fi}
 \def\cprime{$'$} \def\polhk#1{\setbox0=\hbox{#1}{\ooalign{\hidewidth
       \lower1.5ex\hbox{`}\hidewidth\crcr\unhbox0}}} \def\cprime{$'$}
 \providecommand{\bysame}{\leavevmode\hbox to3em{\hrulefill}\thinspace}
 \providecommand{\MR}{\relax\ifhmode\unskip\space\fi MR }
 \providecommand{\MRhref}[2]{%
   \href{http://www.ams.org/mathscinet-getitem?mr=#1}{#2}
 }
 \providecommand{\href}[2]{#2}

\end{document}